\documentclass[11pt,a4paper]{amsart}
\usepackage{amsmath, graphicx, subfigure, epsfig,amssymb,cite}
\usepackage{xcolor,mathtools}
\usepackage[normalem]{ulem}
\numberwithin{equation}{section}

\usepackage{hyperref}
\hypersetup{
    colorlinks=true,
    linkcolor=blue,
    filecolor=magenta,      
    urlcolor=cyan,
    pdftitle={Overleaf Example},
    pdfpagemode=FullScreen,
    }

\newcommand{\cx}{Q}
\newcommand{\e}{\epsilon}
\newcommand{\ga}{\gamma}
\newcommand{\de}{\delta}
\newcommand{\br}{\mathbb{R}}

\newcommand{\ik}{\varphi}
\newcommand{\pa}{\partial}
\newcommand{\bt}{\beta}
\newcommand{\al}{\alpha}
\newcommand{\la}{\lambda}

\newcommand{\be}{\begin{equation}}
\newcommand{\ee}{\end{equation}}

\newcommand{\tal}{\tilde\al}

\newcommand{\dd}{\text{d}}
\newcommand{\frec}{f_\e^{\text{rec}}}

\newcommand{\fprec}{\textcolor{black}{f_\e^{\text{p-rec}}}}

\newcommand{\fme}{\textcolor{black}{f_\e^{\text{m}}}}
\newcommand{\fpe}{\textcolor{black}{f_\e^{\text{p}}}}

\newcommand{\hfe}{\hat f_\e}
\newcommand{\hfpe}{\hat f_\e^{\text{p}}}

\newcommand{\alst}{\alpha_{\star}} 
\newcommand{\TV}{\text{TV}}

\newcommand{\CA}{\mathcal{A}}

\newcommand{\s}{\mathcal S}
\newcommand{\R}{\mathcal R}
\newcommand{\CF}{\mathcal F}

\newcommand{\CH}{\mathcal H}
\newcommand{\BZ}{\mathbb Z}

\newcommand{\oks}{I}
\newcommand{\fks}{\phi_{l}}
\newcommand{\CD}{\mathcal D}

\newcommand{\DTB}{\text{DTB}}

\newcommand{\phmn}{\phi_{\text{mn}}^{\prime\prime}}

\newtheorem{theorem}{Theorem}[section]
\newtheorem{lemma}[theorem]{Lemma}
\newtheorem{corollary}[theorem]{Corollary}

\theoremstyle{definition}
\newtheorem{assumptions}[theorem]{Assumptions}
\newtheorem{definition}[theorem]{Definition}
\newtheorem{remark}[theorem]{Remark}

\begin{document}

\title[Analysis of reconstruction]{Analysis of reconstruction of functions with rough edges from discrete Radon data in $\mathbb R^2$}
\author[A Katsevich]{Alexander Katsevich$^1$}
\thanks{$^1$This work was supported in part by NSF grant DMS-1906361. Department of Mathematics, University of Central Florida, Orlando, FL 32816 (Alexander.Katsevich@ucf.edu). }

\begin{abstract} We study the accuracy of reconstruction of a family of functions $f_\epsilon(x)$, $x\in\mathbb R^2$, $\epsilon\to0$, from their discrete Radon transform data sampled with step size $O(\epsilon)$. For each $\epsilon>0$ sufficiently small, the function $f_\epsilon$ has a jump across a rough boundary $\mathcal S_\epsilon$, which is modeled by an $O(\epsilon)$-size perturbation of a smooth boundary $\mathcal S$. The function $H_0$, which describes the perturbation, is assumed to be of bounded variation. Let $f_\e^{\text{rec}}$ denote the reconstruction, which is computed by interpolating discrete data and substituting it into a continuous inversion formula. We prove that $(f_\e^{\text{rec}}-K_\epsilon*f_\epsilon)(x_0+\epsilon\check x)=O(\epsilon^{1/2}\ln(1/\epsilon))$, where $x_0\in\mathcal S$  and $K_\epsilon$ is an easily computable kernel. 
\end{abstract}

\maketitle

\section{Introduction}

Let $f$ be a function in $\br^2$ and $\s$ be some curve. Suppose $f$ has a jump across $\s$ and $f$ is smooth otherwise. Let $\frec$ be a reconstruction from discrete Radon transform (RT) data of $f$, where $\e$ represents the data step size. The reconstruction is computed by substituting interpolated data into a ``continuous'' inversion formula. In many applications it is important to know how accurately the singularities of $f$ are reconstructed, including medical imaging, materials science, and nondestructive testing.

To address this, in \cite{Katsevich2017a, kat19a, Katsevich2020a, Katsevich2020b, Katsevich2021a}, the author developed a novel method of image reconstruction analysis, which studies the behavior of $\frec$ in an $O(\e)$ neighborhood of $\s$. In those papers, the technique was referred to as \emph{local resolution analysis}. However, this theory supplies more information than just resolution and applies away from $\s$ as well (see e.g. \cite{Katsevich_aliasing_2023}). In light of this, we will henceforth refer to the technique as \emph{Local Reconstruction Analysis} (LRA). The main idea of LRA is to represent $\frec$ in the form 
\be\label{DTB new use}
\frec(x_0+\e\check x)=\DTB(\check x;x_0,\e)+\text{remainder},
\ee
where $x_0\in\s$ and $\check x$ is confined to a bounded set. Thus, {\it we study reconstruction near $\s$ in a neighborhood, whose size decreases at the same rate as the data step size.} Here $\DTB(\check x;x_0,\e)$ is an easily computable function, which we call the discrete transition behavior (or DTB for short), and the remainder is small. Usually, $\DTB(\check x;x_0,\e)$ is computed by convolving the leading singularity of $f$ near $x_0\in \s$ with a simple, explicitly computable kernel. 

The practical use of the DTB function is based on the following idea. When $\e$ is small, the remainder is small, and $\DTB(\check x;x_0,\e)$, which is given by a simple formula (see e.g. \eqref{via_kernel_lim}), is an accurate approximation to $\frec$, which is {\it the actual reconstruction from discrete data.} Numerical experiments in \cite{Katsevich2017a, kat19a, Katsevich2020a, Katsevich2020b} show that the remainder in \eqref{DTB new use} is indeed quite small for realistic values of $\e$. Hence, a wide range of time-consuming analyses of reconstructed images can be replaced by the corresponding analyses of $\DTB(\check x;x_0,\e)$. For example, one can use $\DTB(\check x;x_0,\e)$ to estimate the resolution of reconstruction, optimize data sampling, and perform other required tasks with the help of a theoretical analysis or in a computationally efficient manner (i.e., without simulating or measuring tomographic data and then reconstructing from it).

Initially, LRA was developed for functions for which $\s=\text{sing\,supp}(f)$ is a smooth curve (or hypersurface). This was done in a number of settings, including generalized Radon transforms, general (parametrix-based) inversion  formulas, etc. \cite{Katsevich2017a,kat19a, Katsevich2020a, Katsevich2020b, Katsevich2021a}. 
In many applications discontinuities of $f$ occur across {\it rough} surfaces. Examples include soil and rock imaging, where the surfaces of cracks and pores and boundaries between regions are highly irregular and frequently simulated by fractals \cite{Anovitz2015, GouyetRosso1996, Li2019, soilfractals2000, PowerTullis1991, Renard2004, Zhu2019}. 

In \cite{Katsevich2023a, Katsevich2022a} the author extended LRA to functions with jumps across rough boundaries. In these papers we study nonsmooth boundaries in an asymptotic ($\e\to0$) fashion. Instead of a single $f$, we consider a family of modified functions $\fme:=f-\fpe$, where $\fpe$ is a perturbation. Locally, each $\fme$ has a jump across $\s_\e$, where $\s_\e$ is a small and suitably scaled perturbation of a (fixed) smooth curve $\s$. 
The local parametrization of $\s$ is $\mathcal I\ni u\to y(u)$, where $\mathcal I$ is an interval, $y(u)$ is sufficiently regular, and $y’(u)\not=0$, $u\in\mathcal I$. The corresponding local parametrization of $\s_\e$ is given by $\mathcal I\ni u\to y(u)+\e H_0(u/\e^{1/2},\e)\vec\theta(u)$, where $\vec\theta(u)$ is a unit vector orthogonal to $\s$ at $y(u)$, $u\in \mathcal I$ (see Fig.~\ref{fig:jump} and Fig.~\ref{fig:pert}). The bounded function $H_0$ defines the perturbation $\s\to\s_\e$. The roughness of $\s_\e$ is described by the roughness of $H_0$ as a function of $u$ ($\e>0$ is viewed as a parameter). Our construction ensures that the size of the perturbation is $O(\e)$ in directions normal to $\s$, and the perturbation scales like $\e^{-1/2}$ in directions tangent to $\s$. 

To prove that the remainder in \eqref{DTB new use} is small when $\s_\e$ is rough, an additional assumption is made in \cite{Katsevich2023a, Katsevich2022a}, namely that the family of functions $H_0(\cdot,\e)$ has level sets that are not too dense. This assumption is not too restrictive and allows for fairly nonsmooth $H_0$. Given any $\ga\in (0,1)$, \cite{Katsevich2023a} provides an example of $H_0$, whose level sets are not too dense as required and which is H{\"o}lder continuous with exponent $\ga$, but not H{\"o}lder continuous with any $\ga'>\ga$ on a dense subset of $\br$. 

To prove that the DTB is accurate for rough $\s_\e$, we show that the magnitude of the remainder is $O(\e^{1/2}\ln(1/\e))$ {\it independently} of how rough $\s_\e$ is (i.e. independent of its H{\"o}lder exponent $\ga$) \cite{Katsevich2022a}. The ideas behind the proofs in \cite{Katsevich2023a, Katsevich2022a} are very different from the earlier ones \cite{Katsevich2017a,kat19a, Katsevich2020a, Katsevich2020b, Katsevich2021a}. The latter proofs revolve around the smoothness of $\s=\text{sing\,supp}(f)$. The new proofs are based on the cancellations occurring in certain exponential sums (see \eqref{extra terms I} below). Numerical experiments in \cite{Katsevich2023a, Katsevich2022a} show an excellent match between the actual reconstruction and $\DTB$ even when $\s_\e$ is fractal (which is a case not covered by theory yet).  

Our overarching goal is to extend LRA (i.e., prove that the $\DTB$ is accurate) to as wide a class of perturbations (i.e., functions $H_0$) as possible. For example, the additional assumption about level sets of $H_0$ in \cite{Katsevich2023a, Katsevich2022a} may still be violated in some applications and it would be good to get rid of it. In this paper we generalize LRA to functions $H_0\in L^\infty(\br)$ with bounded variation (see assumptions~\ref{ass:H0} for a precise definition). The proofs are different than those in \cite{Katsevich2023a, Katsevich2022a}. We reduce exponential sums to sums of oscillatory integrals, which are then estimated. This makes the proofs simpler and more transparent.

The paper is organized as follows. The state of the art and significance of our results from a theoretical perspective are discussed in section~\ref{sec:theor motiv}. Practical motivation of this research and potential applications are discussed in section~\ref{sec:motiv}. In section~\ref{sec:a-prelims} we describe the set-up, introduce notation and assumptions, and state the main result Theorem~\ref{main-res}. In section~\ref{sec:beg proof} we reduce the statement of the theorem to two inequalities (see \eqref{extra terms I} and \eqref{extra terms II}). The first of them, which is more difficult to prove, involves exponential sums. We also outline the main steps in the proof of the two inequalities. In the proof we consider the following three cases:
\begin{itemize}
\item[(A)] $x_0\in\s$; 
\item[(B)] $x_0\not\in\s$, and there is a line through $x_0$, which is tangent to $\s$; and 
\item[(C)] $x_0\not\in\s$, and no line through $x_0$ is tangent to $\s$. 
\end{itemize}
The proofs of the two inequalities in case (A) is in sections~\ref{sec:beg A} and \ref{sec:end A}. The proofs of both inequalities in cases (B) and (C) are in sections~\ref{sec:caseB} and \ref{sec:caseC}, respectively. The proofs of auxiliary results are in appendices~\ref{sec:prf sum_int}--\ref{sec:model ints}.

\section{Theoretical motivation} \label{sec:theor motiv}

\subsection{State of the art}
Analysis of reconstruction from discrete Radon transform data is an important topic, which has been studied quite well. The most common approach is based on sampling theory, see e.g. \cite{nat93, pal95, far04} and references therein. The underlying assumption is that the function $f$ being reconstructed is essentially bandlimited, i.e. that its Fourier transform is sufficiently small outside some compact set \cite{far04}. A typical prediction of sampling theory is the optimal data sampling step size that guarantees accurate recovery of $f$. Note that the requirement that $f$ be essentially bandlimited implies $f\in C^\infty$.

Another approach to study reconstruction of a nonsmooth $f$ from discrete data uses tools of semiclassical analysis  \cite{stef20, Monard2021}. In these papers, $f$ is not necessarily bandlimited, but the data represent the (generalized) Radon transform of $f$ convolved with a classically bandlimited mollifier $w$. 
The latter is a model of the detector aperture function. Thus, $w$ is analytic, so it is $C^\infty$ and not compactly supported. In practice, detector aperture functions have limited smoothness. E.g., they are likely to be non-smooth across the boundary of a detector pixel. More importantly, $\text{supp}(w)$ is usually compact and coincides with the support of the detector pixel. Even though a classically bandlimited $w$ can be fairly accurately approximated by a function, which is not bandlimited, this does not eliminate the need for a rigorous theoretical analysis of reconstruction with a realistic $w$.

Computer Tomography (CT) is a hugely popular imaging modality, which has found numerous applications in medical imaging, nondestructive testing, security inspection, and many other areas. CT frequently involves imaging of objects with edges. The functions $f$, which describe such objects, have sharp jumps across curves (in $\br^2$) or surfaces (in $\br^3$), and/or other singularities. Such functions are not bandlimited. Therefore, theoretical analysis of image reconstruction of such objects from discrete data is highly desired. Convergence of reconstruction algorithms in the case of $f$ with discontinuities has been studied in the literature \cite{gonchar86,  popov90, pal93, popov98}. However, in these works the discontinuities of the object are a complicating factor rather than the object of study. 

\subsection{LRA, $WF(f)$, and aliasing}
In LRA, neither $f$ nor $w$ is bandlimited in any way (classically or semiclassically). Since Radon data are discrete, the reconstruction is necessarily a combination of a useful signal (a smoothed singularity of $f$) and aliasing. Comparing with \eqref{DTB new use}, the useful part is described by $\DTB(\check x;x_0,\e)$, and aliasing is a part of the remainder. Hence estimating the magnitude of aliasing is an essential part of our proofs.

Recall that aliasing occurs when $\CF f$, the Fourier transform of $f$, does not decay sufficiently fast at infinity. Consider the behavior of $\CF f$ at infinity for a typical $f$ that has a jump across some curve. Qualitatively, this behavior is best described using the notion of the wave-front set of a function, $WF(f)$ \cite[Section 8.1]{hor}.

Most commonly, a function $f(x)$, $x\in\br^2$, which is non-smooth across a smooth curve $\s\subset\br^2$, is described as a conormal distribution associated with $\s$. For such an $f$ we have $WF(f)\subset N^*\s$, the conormal bundle of $\s$ (see \cite[Example 3.3, Chapter VII and Section 3, Chapter VIII]{trev2}). In other words, $(x_0,\xi_0)\in WF(f)$ only when the directions $\xi_0$ are conormal to $\s$ at $x_0$ (which all lie on a line). This means that the Fourier transform of a suitably localized $f$ decays slowly only in an arbitrarily small conic neighborhood of such $\xi_0$. In contrast, if $f_\e$ has a nonsmooth boundary $\s_\e$, then $WF_{x_0}(f_\e)$,  $x_0\in\s_\e$, generally contains the {\it entire} $\br^2\setminus0$. Consequently, regardless of how well we localize $f_\e$ around $x_0$, its Fourier transform decays slowly along all directions. Hence, there is much more potential for aliasing. This makes LRA, which necessarily includes an analysis of aliasing, more complicated in the case of rough boundaries and necessitates the development of new tools.

Our method of proof illustrates the mechanism by which aliasing from different parts of $\s_\e$ gets reduced when the Radon transform inversion is applied to discrete data, with only the useful contribution (the DTB function) remaining in the limit as $\e\to0$. The mechanism is based on oscillations inherent to the Radon transform (i.e., those described by the exponential factor in \eqref{extra terms I}) interacting with an edge at $\s_\e$ (encoded in $A_m$ \eqref{recon-ker-v2}), see Remark~\ref{rem:key}. LRA is a powerful tool for exploring various aspects of this interaction. 

\section{Practical motivation}\label{sec:motiv} 

\subsection{Applications in petroleum industry}
One application of our results is in micro-CT, which is CT capable of achieving micrometer resolution. Micro-CT is widely used in the petroleum industry to image rock samples extracted from wells. For example, the reconstructed images are used as input into Digital Rock Physics (DRP) analyses. A collection of numerical methods to estimate various rock properties using digital cores is called DRP \cite{Fredrich2014, Saxena2019}. The term ``digital core'' stands for a digital representation of the rock sample (rock core) obtained by micro-CT scanning, reconstruction, and image analysis (segmentation and classification, feature extraction, etc.) \cite{Guntoro2019}. DRP ``is a rapidly advancing technology that relies on digital images of rocks to simulate multiphysics at the pore-scale and predict properties of complex rocks (e.g., porosity, permeability, compressibility). ... For the energy industry, DRP aims to achieve more, cheaper, and faster results as compared to conventional laboratory measurements.'' \cite{Saxena2019}. Furthermore, ``The simulation of various rock properties based on three-dimensional digital cores plays an increasingly important role in oil and gas exploration and development. {\it The accuracy of 3D digital core reconstruction is important for determining rock properties}.''   \cite{Zhu2019} (italic font added).

Consider one example. First, the reconstructed micro-CT image is segmented to identify the pore space inside the sample as accurately as possible. Then numerical fluid flow simulations {\it inside the identified pore space} are used to compute the permeability of the sample \cite{bss18}. Finally, the computed permeability is used for formation evaluation and improving oil recovery \cite{sha17}. 

Clearly, CT image segmentation is a key step that affects the accuracy of fluid flow simulation. Boundaries between regions inside the rock (e.g., between the solid matrix and pore space) are rough \cite{Cherk2000}, making image segmentation an especially challenging task. Errors in pore space identification may lead to incorrect flow simulation and errors in computed permeability \cite{Chhatre2017, Saxena2019, Saxena2019a}. Therefore, precise knowledge of the accuracy and resolution of the reconstructed image (especially near sharp features in the image), which affect the accuracy of image segmentation, is of utmost importance. Effects that degrade the resolution of micro-CT (e.g., due to finite data sampling) and how these effects manifest themselves in the presence of rough boundaries require careful investigation. Once fully understood and quantified, these effects can be accounted for, leading to more accurate simulations and DRP outcomes.

\subsection{Applications in other areas}
Objects with rough boundaries, including fractal boundaries, occur in many other applications of physics, engineering, and technology \cite{Bramb2017}. One of their most consequential uses is in the design of antennas \cite{Kr_2017}. “Recently, the self-filling in space curves like Hilbert and Peano fractals were \dots used to design high performance, low profile, conformal antennas with enhanced radiation characteristics and improved power gain of various communication and radar applications” \cite{Kr_2017}. In \cite{JSP2017} the authors discuss 3D fractal antennas: their “envisaged applications are in the defense and aerospace sectors where high-value, lightweight, and mechanically robust antennas can be integrated with other \dots parts.” Given our communication-centric world and the need for robust manufacturing processes, CT-based nondestructive testing of fractal antennas is extremely important. Because the correct shape of such antennas, down to the smallest details, is important to enforce, their high-resolution CT imaging during and after manufacturing will facilitate maintaining strict quality standards. 

Among the medical applications of CT reconstruction of rough shapes is cancer imaging. Healthy and cancerous tissues grow blood vessels differently (i.e., they have different {\it perfusion} patterns), so analysis of the roughness of boundaries in the reconstructed image may help to detect malignancies. In \cite{MD2014} it was demonstrated that “fractal analysis is a suitable method for quantifying perfusion heterogeneity.” In \cite{MHD2022} it was established for one type of pancreatic cancer that “fractal analysis visualizes perfusion chaos in the tumor rim and improves size measurement on CT in comparison to” other methods. Thus, more accurate CT reconstruction of rough shapes may lead to better detection and a more accurate assessment of tumors.

\section{Preliminaries}\label{sec:a-prelims}

Consider a real-valued function $f(x)$, $x\in\br^2$, in the plane, and let $\s$ be some curve. 
For a domain $U\subset\br^n$, $n=1,2,\dots$, the notation $f\in C^k(U)$ means that $f(x)$ is $k$ times differentiable and $\pa_x^m f \in L^\infty(U)$ for any multi-index $m$, $|m|\le k$. The notation $f\in C_0^k(U)$ means that $f\in C^k(U)$ and $\text{supp}(f)\subset U$.

\begin{assumptions}[Properties of $f$]\label{ass:f}
$\hspace{1cm}$ 
\begin{itemize}
\item[F1.] $\s$ is a $C^4$ curve;
\item[F2.] $f$ is compactly supported and $f\in C^2(\br^2\setminus\s$); and
\item[F3.] For each $x_0\in\s$ there exist a neighborhood $\mathcal U\ni x_0$, domains $D_\pm$, and functions $f_\pm\in C^2(\br^2)$ such that
\begin{equation}\label{f_def}\begin{split}
& f(x)=\chi_{D_-}(x) f_-(x)+\chi_{D_+}(x) f_+(x),\ x\in \mathcal U\setminus \s,\\
& D_-\cap D_+=\varnothing,\ D_-\cup D_+=\mathcal U\setminus \s,
\end{split}
\end{equation}
where $\chi_{D_\pm}$ are the characteristic functions of $D_\pm$.
\end{itemize}
\end{assumptions}
Assumptions~\ref{ass:f} describe a typical function, which has a jump discontinuity across a smooth curve (see Figure~\ref{fig:jump}).

\begin{figure}[h]
{\centerline
{\epsfig{file=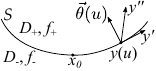, width=5.0cm}}}
\caption{Illustration of a function $f$ discontinuous across $\s$.}
\label{fig:jump}
\end{figure}

Discrete Radon transform data are given by
\be\label{data_eps}
\hfe(\al_k,p_j):=\frac1\e \iint_{\br^2} w\left(\frac{p_j-\vec\al_k\cdot y}{\e}\right)f(y)\dd y,\
p_j=\bar p+j\Delta p,\ \al_k=\bar\al+k\Delta\al,
\ee
where $w$ is the detector aperture function, $\Delta p=\e$, $\Delta\al=\kappa\e$, $\kappa>0$, and $\bar p$, $\bar\al$, are fixed. 
Here and below, $\vec \al$ and $\al$ in the same equation are related by $\vec\al=(\cos\al,\sin\al)$. The same applies to $\vec\theta=(\cos\theta,\sin\theta)$ and $\theta$. Sometimes we also use $\vec\al^\perp=(-\sin\al,\cos\al)$, and similarly for $\theta^\perp$. 

\begin{assumptions}[Properties of the aperture function $w$]\label{ass:w}
$\hspace{1cm}$
\begin{enumerate}
\item $w$ is even, $w\in C_0^{\lceil \bt \rceil+1}(\br)$; and 
\item Normalization: $\int w(p)dp=1$.
\end{enumerate}
\end{assumptions}
Here $\lceil \bt\rceil$ is the ceiling function, i.e. the integer $n$ such that $n-1<\bt\le n$. The required value of $\bt$ is stated below in Theorem~\ref{main-res}. Later we also use the floor function $\lfloor \bt\rfloor$, which gives the integer $n$ such that $n\le\bt<n+1$, and the fractional part function $\{\bt\}:=\bt-\lfloor \bt\rfloor$. 

Reconstruction from the data \eqref{data_eps} is achieved by the formula
\be\label{recon-orig}
\frec(x)=-\frac{\Delta\al}{2\pi}\sum_{|\al_k|\le \pi/2} \frac1\pi \int \frac{\pa_p\sum_j\ik\left(\frac{p-p_j}\e\right)\hfe(\al_k,p_j)}{p-\al_k\cdot x}\dd p,
\ee
where $\ik$ is an interpolation kernel.

\begin{assumptions}[Properties of the interpolation kernel $\ik$]\label{ass:ik}
$\hspace{1cm}$
\begin{enumerate}
\item $\ik$ is even, compactly supported, and its Fourier transform, $\tilde\ik$, satisfies $\tilde\ik(\la)=O(|\la|^{-(\bt+1)})$, $\la\to\infty$;
\item $\ik$ is exact up to order $1$, i.e. 
\be\label{exactness}
\sum_{j\in\mathbb Z} j^m\ik(u-j)\equiv u^m,\ m=0,1,\ u\in\br.
\ee
\end{enumerate}
\end{assumptions}
Here $\bt$ is the same as in assumption~\ref{ass:w}(1). Assumption~\ref{ass:ik}(2) implies $\int\ik(p)dp=1$. See Section IV.D in \cite{btu2003}, which shows that $\ik$ with the desired properties can be found for any $\bt>0$ (i.e., for any regularity of $\ik$).

For a function $f$ on an interval $[a,b]$, the total variation of $f$ is defined as follows:
\be\label{var def}
\TV(f,[a,b]):=\sup\sum_{i=1}^N|f(t_i)-f(t_{i-1})|,\ a\le t_0<t_1<\dots<t_N\le b,
\ee
where the supremum is taken over $N\ge 1$ and all corresponding partitions of the interval \cite[Section 6.3]{Royden2010}.

Consider a local segment of $\s$, which is parametrized by $\mathcal I\ni u\to y(u)\in\s$, where $\mathcal I$ is an interval, $y\in C^4(\mathcal I)$, and $|y’(u)|\ge c>0$ for any $u\in \mathcal I$. Let $\vec\theta(u)$ be a unit normal to $\s$ at $y(u)$ (continuously defined), see Figure~\ref{fig:jump}. We introduce a family of local perturbations, $\fpe$, $\e>0$, of $f$ so that for each $\e>0$ the modified function $\fme:=f-\fpe$ has a jump across some rough curve $\s_\e$, which is a local perturbation of $\s$. 

Let $H_0(u;\e)$, $u\in\br$, be a family of real-valued functions defined for all $\e>0$ sufficiently small. The local perturbations $\fpe$ are constructed as follows (see Figure~\ref{fig:pert}, left panel)
\be\label{main-fn}\begin{split}
\fpe(x):=&(f_+(x)-f_-(x))\chi(t,H_\e(u)),\ 
\chi(t,r):=\begin{cases} 1,&0< t\le r,\\
-1,&r\le t < 0,\\
0,&\text{otherwise},
\end{cases}\\
x=&y(u)+t\vec\theta(u),\ H_\e(u):=\e H_0(\e^{-1/2}u;\e),\ u\in \mathcal I,
\end{split}
\ee
where $f_\pm$ are the functions appearing in \eqref{f_def} (see also Figure~\ref{fig:jump}).

As is easily seen, the modified function $\fme=f-\fpe$ has a jump across a curve $\s_\e$ (instead of $\s$), where $\s_\e$ is parametrized by (see Figure~\ref{fig:pert}, right panel)
\be\label{se param}
\mathcal I\ni u\to y(u)+H_\e(u)\vec\theta(u)\in\s_\e.
\ee
Suppose $H_0$ has the following properties:  

\begin{assumptions}[Properties of local perturbations $H_0$]\label{ass:H0}
$\hspace{1cm}$
\begin{enumerate}
\item There exists a constant, denoted $\Vert H_0\Vert$, such that $\sup_{u\in\br}|H_0(u;\e)|\le \Vert H_0\Vert$ for all $\e>0$ sufficiently small;
\item There exists a constant, denoted $V_0$, such that $\TV(H_0,I)\le V_0|I|$ for any interval $I\subset\br$ and all $\e>0$ sufficiently small.
\end{enumerate}
\end{assumptions}

By assumption~\ref{ass:H0}(1), $\s_\e$ is indeed an $O(\e)$-size perturbation of $\s$. 
We do not restrict what fraction of $\s$ is perturbed. For example, all of $\s$ can be perturbed. When more than one local segment of $\s$ is perturbed, we assume that each local perturbation satisfies assumptions~\ref{ass:H0}, the functions $H_0$ in different neighborhoods of $\s$ can be different, and different versions of $H_0$ are compatible whenever two neighborhoods intersect.   


\begin{figure}[h]
{\centerline{
{\hbox{
{\epsfig{file=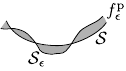, width=4.0cm}}
{\hspace{5mm}}
{\epsfig{file=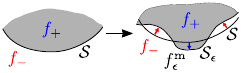, width=7cm}}
}}}}
\caption{Illustration of the perturbation $\fpe(x)$ (left panel) and its construction via the perturbed boundary $\s_\e$ (right panel).}
\label{fig:pert}
\end{figure}

By \eqref{data_eps}, \eqref{recon-orig},
\be\label{via_kernel}\begin{split}
\fprec(x)=-\frac{\Delta\al}{2\pi}\frac1{\e^2}\sum_{|\al_k|\le \pi/2}\sum_j & \CH\ik'\left(\frac{\vec\al_k\cdot x-p_j}\e\right)\\
&\times \iint w\left(\frac{p_j-\vec\al_k\cdot y}{\e}\right)\fpe(y)\dd y,
\end{split}
\ee
where $\CH$ denotes the Hilbert transform (the integral with respect to $p$ in \eqref{recon-orig}).
Following \cite{Katsevich2023a, Katsevich2022a}, replace the sums with respect to $k$ and $j$ in \eqref{via_kernel} with integrals to formally obtain:
\be\label{via_kernel_lim}\begin{split}
\fprec(x)\approx&\frac{1}{\e^2}\iint K\left(\frac{x-y}\e\right)\fpe(y)\dd y,\\
K(z):=&-\frac{1}{2\pi}\int_0^\pi (\CH\ik'*w)(\vec\al\cdot z) \dd\al.
\end{split}
\ee
Using that $\ik$ and $w$ are even it is easy to see that the Radon transform of $K$, $\hat K$, satisfies
\be
\hat K(\al,p)=(\ik * w)(p),\ |\al|\le\pi,p\in\br,
\ee
so $K$ is radial and compactly supported. 

The following definition is in \cite[p. 121]{KN_06} (after a slight modification in the spirit of \cite[p. 172]{Naito2004}).

\begin{definition} Let $\eta>0$. The irrational number $s$ is said to be of type $\eta$ if for any $\eta_1>\eta$, there exists $c(s,\eta_1)>0$ such that
\be\label{type ineq}
m^{\eta_1}\langle ms\rangle \ge c(s,\eta_1) \text{ for any } m\in\mathbb N.
\ee
\end{definition}
See also \cite{Naito2004}, where the numbers which satisfy \eqref{type ineq} are called $(\eta-1)$-order Roth numbers. It is known that $\eta\ge1$ for any irrational $s$. The set of irrationals of each type $\eta \ge 1$ is of full measure in the Lebesgue sense \cite{Naito2004}.

\begin{definition}\label{def:gp} A point $x_0\in\br^2$ is said to be generic if the following assumptions are satisfied:
\begin{enumerate}
\item No line, which is tangent to $\s$ at a point where the curvature of $\s$ is zero, passes through $x_0$;
\item The line through the origin and $x_0$ is not tangent to $\s$; 
\item The quantity $\kappa|x_0|$ is irrational and of some finite type $\eta_\star$; and
\item If $x_0\in\s$, then the quantity $\kappa \vec\al_0^\perp\cdot x_0$, where $\al_0^\perp$ is a unit vector tangent to $\s$ at $x_0$, is irrational and of some finite type $\eta_0$.
\end{enumerate}
\end{definition}

\begin{figure}[h]
{\centerline
{\epsfig{file={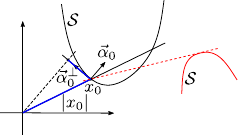}, width=7cm}}}
\caption{Illustration of Definition~\ref{def:gp}. The red segment of $\s$ and the red dashed line illustrate the situation prohibited by (1). The two blue line segments are involved in (3) and (4).}
\label{fig:x0}
\end{figure}

See Figure~\ref{fig:x0} for an illustration of Definition~\ref{def:gp}. In the rest of the paper, we consider only generic points $x_0$. Clearly, the set of generic $x_0$ is dense in the plane. 

By linearity of the measurement process described by \eqref{data_eps} and the inversion formula \eqref{recon-orig} (which correspond to the steps $f\to \hfe$ and $\hfe\to \frec$, respectively), we can ignore the original function $f$ (which has jumps only across smooth curves) and consider the reconstruction of only the perturbation $\fpe$: $\fpe\to\hfpe\to\fprec$. Therefore, henceforth we consider only the reconstruction of $\fpe$. 
Our main result is the following theorem. 

\begin{theorem}\label{main-res} Suppose 
\begin{enumerate}
\item $f$ satisfies assumptions \ref{ass:f}; 
\item the detector aperture function $w$ satisfies assumptions \ref{ass:w}; 
\item the interpolation kernel $\ik$ satisfies assumptions \ref{ass:ik}; 
\item in a neighborhood of any $x\in\s$, the perturbation $H_0$ satisfies assumptions \ref{ass:H0}; and 
\item the point $x_0$ is generic. 
\end{enumerate}
If either \textcolor{black}{$\bt>\max(\eta_0+2,\eta_\star+1)$} and $x_0\in\s$ or \textcolor{black}{$\bt>\max(\eta_0+1,(\eta_\star/2)+1)$} and $x_0\not\in\s$, one has:
\be\label{full conj}
\fprec(x_0+\e\check x)=\frac1{\e^2}\iint K\left(\frac{(x_0+\e\check x)-y}\e\right)\fpe(y)\dd y+O(\e^{1/2}\ln(1/\e)),\ \e\to0,
\ee 
where $K$ is given by \eqref{via_kernel_lim} and the big-$O$ term is uniform with respect to $\check x$ in any compact set. 
\end{theorem}

Since $K$ is compactly supported, we get the following corollary. 
\begin{corollary}\label{main-cor} Under the assumptions of Theorem~\ref{main-res}, if $x_0\not\in\s$, one has
\be\label{away res}
\fprec(x_0+\e\check x)=O(\e^{1/2}\ln(1/\e)),\ \e\to0,
\ee 
where the big-$O$ term is uniform with respect to $\check x$ in any compact set. 
\end{corollary}

The estimate in Corollary \ref{main-cor} is fairly tight, since aliasing artifacts away from a jump of $f$, even when the jump is across a smooth curve, are at least of magnitude $\sim \e^{1/2}$ \cite{Katsevich_aliasing_2023}.

\section{Reduction of Theorem~\ref{main-res} to two inequalities}\label{sec:beg proof}

Due to the linearity of the Radon transform and the inversion formula we may assume that $\text{supp}(f)$ is sufficiently small and the curve $\s$ is sufficiently short. Therefore, the support of $\fpe$ (along $\s$) is small, and $\s_\e$ is short as well. For convenience, we state two local versions of our main result.

\begin{lemma}\label{lem:main-res-A} Suppose that the assumptions of Theorem~\ref{main-res} are satisfied. Suppose the support of $\fpe$ is sufficiently small, $\s$ is sufficiently short, and case (A) holds. If \textcolor{black}{$\bt>\max(\eta_0+2,\eta_\star+1)$,} then \eqref{full conj} holds, where $K$ is given by \eqref{via_kernel_lim}, and the big-$O$ term is uniform with respect to $\check x$ in any compact set. 
\end{lemma}

\begin{lemma}\label{lem:main-res-BC} Suppose that the assumptions of Theorem~\ref{main-res} are satisfied. Suppose the support of $\fpe$ is sufficiently small, $\s$ is sufficiently short, and $x_0\not\in\s$. Suppose \textcolor{black}{$\bt>\eta_0+1$} if case (B) holds and \textcolor{black}{$\bt>\max(2,(\eta_\star/2)+1)$} if case (C) holds. Then \eqref{full conj} holds, where $K$ is given by \eqref{via_kernel_lim}, and the big-$O$ term is uniform with respect to $\check x$ in any compact set.
\end{lemma}

Obviously, Theorem~\ref{main-res} follows from the above two lemmas.

Following \cite{Katsevich2023a, Katsevich2022a}, introduce the function (see \eqref{via_kernel})
\be\label{recon-ker-0}
\psi(q,t):=\sum_j (\CH\ik')(q-j)w(j-q-t).
\ee
Then
\be\label{psi-props-0}\begin{split}
&\psi(q,t)=\psi(q+1,t),\  |\psi(q,t)|\le c(1+t^2)^{-1},\ q,t\in\br.
\end{split}
\ee
By \eqref{psi-props-0}, we can represent $\psi$ in terms of its Fourier series:
\be\label{four-ser}\begin{split}
\psi(q,t)&=\sum_m \tilde\psi_m(t) e(-mq),\ e(q):=\exp(2\pi i q),\\
\tilde\psi_m(t) &=\int_0^1 \psi(q,t)e(mq)\dd q=\int_\br (\CH\ik')(q) w(-q-t)e(mq)\dd q.
\end{split}
\ee

Introduce the function $\rho(s):=(1+|s|)^{-\bt}$, $s\in\br$. For convenience, throughout the paper we use the following convention. If a constant $c$ is used in an equation, the qualifier ‘for some $c>0$’ is assumed. If several $c$ are used in a string of (in)equalities, then ‘for some’ applies to each of them, and the values of different $c$’s may all be different. 

\begin{lemma}[\cite{Katsevich2022a}]\label{lem:psi psider}
Under assumptions~\ref{ass:w} and \ref{ass:ik}, one has
\be\label{four-coef-bnd}
|\tilde\psi_m(t)|,|\tilde\psi_m'(t)| \le c\rho(m)(1+t^2)^{-1},\ t\in\br,m\in\mathbb Z.
\ee
\end{lemma}
By the last lemma, the Fourier series for $\psi$ converges absolutely. From \eqref{data_eps}, \eqref{via_kernel}, \eqref{recon-ker-0},  and \eqref{four-ser}, the reconstructed image becomes
\be\label{recon-ker-v2}
\begin{split}
&\textcolor{black}{\fprec}(x)
=-\frac{\Delta\al}{2\pi}\sum_m \sum_{|\al_k|\le \pi/2} e\left(-m \frac{\vec\al_k\cdot x-\bar p}\e\right)A_m(\al_k,\e), \\
&A_m(\al,\e):=\e^{-2}\iint \tilde\psi_m\left(\frac{\vec\al\cdot (y-x)}\e\right)\textcolor{black}{\fpe}(y)\dd y.
\end{split}
\ee
We will show that 
\begin{align}\label{extra terms I}
&\Delta\al\sum_{m\not=0}\bigg| \sum_{|\al_k|\le\pi/2} e\left(-m \frac{\vec\al_k\cdot x}\e\right)A_m(\al_k,\e)\bigg|=O(\e^{1/2}\ln(1/\e)),
\end{align}
and
\be\label{extra terms II}
\sum_{|\al_k|\le\pi/2} \int_{|\al-\al_k|\le \Delta\al/2}|A_0(\al,\e)-A_0(\al_k,\e)|\dd\al=O(\e^{1/2}\ln(1/\e)),
\ee
where $x=x_0+\e\check x$. The factor $e(m\bar p/\e)$ is dropped because it is independent of $k$.
The main steps in the proof of \eqref{extra terms I} in all three cases are as follows:
\begin{enumerate}
\item Simplify the expression for $A_m(\al,\e)$;
\item Establish generic rates of decay of $A_m(\al,\e)$ for $\al$ away from $\al=0$ and as $m\to\infty$;
\item Using the established rates, identify a set of $m$ and $k$ values, whose contribution to the sum in \eqref{extra terms I} is of magnitude $O(\e^{1/2}\ln(1/\e))$;
\item Show that the sum with respect to the remaining $k$ and $m$ can be approximated by a sum of some integrals, and the cumulative error of the approximation is of magnitude 
$O(\e^{1/2}\ln(1/\e))$. In this step we estimate $\TV(\e A_m(\al,\e),I)$ as $\e\to0$ for various intervals $I\subset [-\pi/2,\pi/2]$;
\item Using integration by parts, show that the cumulative contribution of the integrals is of magnitude 
$O(\e^{1/2}\ln(1/\e))$.
\end{enumerate}

The proof of \eqref{extra terms II} uses steps (1)--(3), and the desired result follows from an estimate of $\TV(\e A_m(\al,\e),[-\pi/2,\pi/2])$ obtained in step (4).

\begin{remark}\label{rem:key} The proof of \eqref{extra terms I} is more difficult, so we concentrate on it first. This proof revolves around a key difference between estimating exponential sums and oscillatory integrals. The main contributions to the latter generally come only from stationary points of the phase, while the main contributions to the former come from points where the derivative of the phase is an integer. For these contributions to be sufficiently small (so that \eqref{extra terms I} holds) requires the coefficients $A_m(\al,\e)$ to be sufficiently nice (e.g., $\TV(\e A_m(\al,\e),[-\pi/2,\pi/2])$ to be small), and the latter property is closely related to how rough the perturbation $\s\to\s_\e$ is. On the other hand, the less regular $H_0$ is, the larger the exponential sums may become. The proofs in this paper strike a balance between the regularity of $H_0$ and the magnitude of $\TV(\e A_m(\al,\e),[-\pi/2,\pi/2])$, $m\in\mathbb Z$.
\end{remark}

\section{Beginning of proof of \eqref{extra terms I} in case (A)}\label{sec:beg A}

By using a partition of unity, if necessary, we can make the following additional assumptions.
\begin{assumptions}\label{ass:S_A} 
$\hspace{1cm}$
\begin{enumerate}
\item $\s$ is a sufficiently short segment, i.e. $\mathcal I=[-a,a]$, where $a>0$ is sufficiently small;
\item $\s$ is parametrized as follows $[-a,a]\ni \theta\to y(\theta)\in\s$, where the interior unit normal to $\s$ at $y(\theta)$ is $\vec\theta=(\cos\theta,\sin\theta)$, $x_0=y(0)$ and $\vec\theta^\perp\cdot y’(\theta)\equiv -|y’(\theta)|<0$, $|\theta|\le a$; 
\item $\s$ has nonzero curvature, $\R(\theta)$, at each point $y(\theta)$, $|\theta|\le a$.
\end{enumerate}
\end{assumptions}

The word ‘interior’ in assumption~\ref{ass:S_A}(2) means that $\vec\theta$ points towards the center of curvature of $\s$ at $y(\theta)$ and $\vec\theta\cdot y^{\prime\prime}(\theta)\equiv -\vec\theta^\perp\cdot y’(\theta)>0$, $|\theta|\le a$. See Figure~\ref{fig:jump} (with $u$ replaced by $\theta$) for an illustration of $\s$ in case (A). 

Transform the expression for $A_m$ in \eqref{recon-ker-v2} by changing variable $y\to (\theta,t)$, where $y=y(\theta)+t\vec\theta$:
\be\label{A-simpl}
\begin{split}
A_m(\al,\e)=&\frac1{\e^2}\int_{-a}^a\int_0^{H_\e(\theta)}\tilde\psi_m\left(\frac{\vec\al\cdot (y(\theta)-x_0)}\e-\vec\al\cdot\check x+(t/\e)\cos(\theta-\al)\right)\\
&\times F(\theta,t)\dd t\dd\theta,\quad F(\theta,t):=\fpe(y(\theta)+t\vec\theta) (\R(\theta)-t),\ 
\end{split}
\ee
where $\R(\theta)-t=\text{det}(\dd y/\dd (\theta,t))>0$. Here $\R(\theta)$ is the radius of curvature of $\s$ at $y(\theta)$, and $0<\R(\theta)<\infty$, $|\theta|\le a$, by Definition~\ref{def:gp}(1).

Change variable $t=\e\hat t$ in \eqref{A-simpl}:
\be\label{A simpl orig}
\begin{split}
A_m(\al,\e)=&\e^{-1}\int_{-a}^a\int_0^{H_0(\e^{-1/2}\theta)}\tilde\psi_m\left(\e^{-1}R(\theta,\al)+h(\hat t,\theta,\al)\right)
F(\theta,\e \hat t)\dd\hat t\dd\theta,\\
R(\theta,\al):=&\vec\al\cdot (y(\theta)-x_0),\ h(\hat t,\theta,\al):=-\vec\al\cdot\check x+\hat t\cos(\theta-\al),
\end{split}
\ee
By assumption~\ref{ass:S_A}(2), $x_0=y(0)$, so the function $R$ becomes
\be\label{bigR1}
R(\theta,\al)=\vec\al\cdot \left(y(\theta)-y(0)\right).
\ee
The convexity of $\s$ and our convention (see assumptions~\ref{ass:S_A}(2, 3)) imply that the nonzero vector $(y(\theta)-y(0))/\theta$ rotates counter-clockwise as $\theta$ increases from $-a$ to $a$ (see Figure~\ref{fig:A_fn}). Thus, for each $\theta\in[-a,a]$ there is $\al=\CA(\theta)\in(-\pi/2,\pi/2)$ such that $\vec\al(\CA(\theta))\cdot \left(y(\theta)-y(0)\right)\equiv0$ and the function $\CA(\theta)$ is injective. By continuity, $\CA(0):=0$.

\begin{figure}[h]
{\centerline
{\epsfig{file=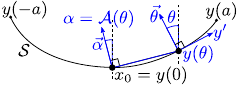, width=6.0cm}}}
\caption{Construction of the function $\CA(\theta)$.}
\label{fig:A_fn}
\end{figure}

\begin{definition}
We say $f(x)\asymp g(x)$ for $x\in U\subset \br^n$ if there exist $c_{1,2}>0$ such that
\be
c_1\le f(x)/g(x)\le c_2\text{ if }g(x)\not=0 \text{ and }f(x)=0\text{ if }g(x)=0
\ee 
for any $x\in U$.
\end{definition}

For convenience, we state here some parts of \cite[Lemma 5.2]{Katsevich2022a} that are needed in this paper. 
Denote $\CD:=[-a,a]\times[-\pi/2,\pi/2]$.
\begin{lemma}\label{lem:aux props}
Under assumptions~\ref{ass:S_A}, one has 
\be\label{r props II}
\CA(\theta)\asymp\theta,\ \CA'(\theta)\asymp1,\ |\theta|\le a;
\ee
and
\be\label{r props III}
\begin{split}
&R(\theta,\al)\asymp \theta(\CA(\theta)-\al),\ 
(\theta,\al)\in\CD.
\end{split}
\ee
\end{lemma}

\begin{lemma}\label{lem:first int} One has
\be\label{model int 0}
\int_{\br}\bigl[1+(\theta(\theta-\al)/\e)^2\bigr]^{-1}\dd\theta\le c\e^{1/2}(1+\e^{-1/2}|\al|)^{-1},\
\al\in\br,\e>0.
\ee
\end{lemma}

Lemma~\ref{lem:first int} is proven in subsection~\ref{ssec:first int}. In \eqref{model int 0}, $\theta-\al$ models the difference $\CA(\theta)-\al$, and $\theta(\theta-\al)$ models $R(\theta,\al)$. By \eqref{four-coef-bnd}, the integral in \eqref{model int 0} models the integral in \eqref{A simpl orig} (up to various factors, e.g. $1/\e$). To more accurately reflect \eqref{r props II}, \eqref{r props III}, we could have used $\theta(r\theta-\al)$ with some specific $r>0$ to model $R(\theta,\al)$. However, the integral $\int_{\br}\bigl[1+(\theta(r\theta-\al)/\e)^2\bigr]^{-1}\dd\theta$ and the one in \eqref{model int 0} are of the same order of magnitude for any $r>0$. Therefore, for simplicity, we set $r=1$.

By combining Lemma~\ref{lem:aux props} and \eqref{model int 0}, we find from \eqref{extra terms I}, \eqref{A simpl orig} 
\be\begin{split}\label{extra extra}
\e&\sum_{|m|\ge \e^{-\ga}} \sum_{|\al_k|\le\pi/2} |A_m(\al_k,\e)|\\
&\le c\sum_{|m|\ge \e^{-\ga}}\rho(m)\sum_{|k|\le O(1/\e)} \frac{\e^{1/2}}{1+\e^{-1/2}|\al_k|}\\
&=O(\e^{1/2}\ln(1/\e)),\ \ga:=1/(2(\bt-1)).
\end{split}
\ee
Therefore, in what follows we always assume $|m|\le \e^{-\ga}$.

\section{End of proof of \eqref{extra terms I} and proof of \eqref{extra terms II} in case (A).}
\label{sec:end A}

\subsection{Two auxiliary lemmas}\label{ssec: two aux lems}
For an interval $I$ and a $C^2(I)$ function $\phi$ define
\be\label{mxmn}
\phi_{\text{mx}}^{\prime\prime}(I):=\sup_{\al\in I}|\phi^{\prime\prime}(\al)|,\
\phi_{\text{mn}}^{\prime\prime}(I):=\inf_{\al\in I}|\phi^{\prime\prime}(\al)|.
\ee
Most of the time we omit the dependence of the above quantities on $I$, because the interval in question is clear from the context.

Recall that $\kappa=\Delta\al/\Delta p$. The following lemma is proven in appendix~\ref{ssec:prf sum1 est}.

\begin{lemma}\label{lem:sum1 est}
Let $I$ be an interval. Pick two functions $\phi$ and $g$ such that $\phi\in C^2(I)$,  $\phi_{\text{mx}}^{\prime\prime}<\infty$; and $g\in BV(I)$. Suppose that for some $l\in \mathbb Z$ one has
\be\label{Omega_eps}
|\kappa  \phi’(\al)-l|\le 1/2\text{ for any }\al\in I.
\ee
Denote 
\be\label{h-fn}\begin{split}
&\fks(\al):=\phi(\al)-(l/\kappa)\al,\\
&h_l(\al):=\frac{\pi\kappa \fks’(\al)}{\sin(\pi\kappa \fks’(\al))}\text{ if } \phi_l’(\al)\not=0,
\ \al\in I,
\end{split}
\ee
and $h_l(\al):=1$ if $\phi_l’(\al)=0$. For all $\e>0$ sufficiently small, one has
\be\label{sum est}\begin{split}
&\biggl|\sum_{\al_k\in I}g(\al_k)e\biggl(\frac{\phi(\al_k)}\e\biggr)-\frac1{\Delta\al}\int_{I} g(\al)h_l(\al) e\biggl(\frac{\fks(\al)}\e\biggr)\dd \al\biggr|\\
&\le c\biggl[\bigl(1+\e\phi_{\text{mx}}^{\prime\prime}\bigr)\TV(g,I)+\phi_{\text{mx}}^{\prime\prime}\int_{I}|g(\al)|\dd\al \biggr],
\end{split}
\ee
where the constant $c$ is independent of $\e$, $\phi$,  $g$, and $I$.
\end{lemma}

In what follows, we use Lemma~\ref{lem:sum1 est} to replace exponential sums, like those in \eqref{extra terms I}, with integrals. Hence, we consider the following integrals with various limits $a_1,a_2$:
\be\label{int jl*}
\frac1{\Delta\al}\int_{a_1}^{a_2} g_\e(\al)h_l(\al) e\biggl(\frac{\fks(\al)}\e\biggr)\dd \al.
\ee
The following lemma is proven in appendix~\ref{lem:int1 est prf}.

\begin{lemma}\label{lem:int1 est} Fix any integer $m\not=0$ and set $\phi(\al)=-m \vec\al\cdot x_0$, where $x_0\in\s$ is generic. Pick any integer $l\in\kappa\phi’([-\pi/2,\pi/2])$, and let $|\al_l^*|\le \pi/2$ be such that $\kappa\phi’(\al_l^*)=l$. Pick any interval $I_l$, $\al_l^*\in I_l\subset[-\pi/2,\pi/2]$, such that 
\begin{align}\label{Omega_eps v2}
&|\kappa  \phi’(\al)-l|\le 1/2,\ \forall\al\in I_l \text{ and }\phi_{\text{mn}}^{\prime\prime}>0.
\end{align}
Suppose $g_\e(\al)\in C^1(I_l)$, $\e>0$, is a family of functions which satisfy
\begin{align}\label{g_e ass 1}
&|g_\e(\al)|\le c\rho(m)\e^{1/2}(1+\e^{-1/2}|\al|)^{-1},\ \forall\al\in I_l,\\ 
&\label{g_e ass 2}
\TV(g_\e,I_\e’)\le c\rho(m)\e^{1/2}(1+\e^{-1/2}|\al|)^{-1},\ \forall \al\in I_\e’,
\end{align}
for any interval $I_\e’\subset I_l$, $|I_\e’|=O(\e^{1/2})$. One has
\be\label{int1 est}\begin{split}
&\biggl|\frac1{\Delta\al}\int_{I_l} (g_\e h_l)(\al) e\biggl(\frac{\phi_l(\al)}\e\biggr)\dd \al\biggr|
\le c \frac{\rho(m)\e^{1/2}}{|\al_l^*|}\biggl[1+\frac{\ln(1/\e)}{\phi_{\text{mn}}^{\prime\prime}}+\frac{|m|}{(\phi_{\text{mn}}^{\prime\prime})^2}\biggr], m\not=0,
\end{split}
\ee
where $c>0$ is an absolute constant.
\end{lemma}

Since $|\kappa\phi’(0)|=|m \kappa \al_0^\perp \cdot x_0|$ is irrational (see assumption (4) in Definition~\ref{def:gp}), $\al_l^*\not=0$ for any $l$ and $m$.

Straightforward application of conventional methods for estimating integrals, e.g. the stationary phase method or integration by parts, is not sufficient. A more delicate use of these techniques is required for our purposes.

\subsection{End of proof of \eqref{extra terms I}.}\label{ssec:end of prf}
Set
\be\label{Am small g}
g_{m,\e}(\al):=e\left(-m \vec\al\cdot \check x\right)\e A_m(\al,\e),\
\phi(\al):=-m \vec\al\cdot x_0.
\ee
Pick any interval $I\subset[-\pi/2,\pi/2]$ and consider sums of the kind
\be\label{generic sum}
J_\e(m,I):=\sum_{\al_k\in I} g_{m,\e}(\al_k) e(\phi(\al_k)/\e).
\ee
From \eqref{extra terms I}, \eqref{extra extra} and \eqref{Am small g} it follows that we need to prove that 
\be\label{Jme sum A}
\sum_{1\le|m|\le\e^{-\ga}}|J_\e(m,[-\pi/2,\pi/2])|=O(\e^{1/2}\ln(1/\e)).
\ee
This is proven using the following steps.
\begin{enumerate}
\item Pick a collection of finitely many nonoverlapping intervals $I_l$ so that $\cup_l I_l=[-\pi/2,\pi/2]$ and, therefore, $\sum_l J_\e(m,I_l)=J_\e(m,[-\pi/2,\pi/2])$;
\item Use Lemma~\ref{lem:sum1 est} with $g=g_{m,\e}$, where $g_{m,\e}$ is given by \eqref{Am small g}, to show that for each $I_l$ the sum \eqref{generic sum} can be replaced by the corresponding integral \eqref{int jl*} over $I_l$. More precisely, we show that the total error of replacing all the sums with the integrals is of order $O(\e^{1/2}\ln(1/\e))$;
\item Use Lemma~\ref{lem:int1 est} with $g_\e=g_{m,\e}$ to estimate the integrals \eqref{int jl*} over $I_l$ and show that the sum of the estimates over all $I_l$ and all $1\le|m|\le\e^{-\ga}$ is $O(\e^{1/2}\ln(1/\e))$.
\end{enumerate}

Thus, our first goal is to apply Lemma~\ref{lem:sum1 est} to \eqref{generic sum}, where $g_{m,\e}$ are given by \eqref{Am small g}. This is achieved by the following lemma, which is proven in appendix~\ref{sec:prf appl lem1}.

\begin{lemma}\label{lem:apply lemma 1} Pick any interval $I\subset[-\pi/2,\pi/2]$. Let $g_{m,\e}$ and $\phi$, where $m\in\mathbb Z$, be defined by \eqref{Am small g}. Under the assumptions of Theorem~\ref{main-res} and assumptions~\ref{ass:S_A} one has $g_{m,\e}\in C^1([-\pi/2,\pi/2])$ and
\begin{align}\label{g_e ass 1 v2}
&|g_{m,\e}(\al)|\le c\rho(m)\e^{1/2}(1+\e^{-1/2}|\al|)^{-1},\ \forall\al\in I,\\ 
&\label{g_e ass 2 v2}
\TV(g_{m,\e},U)\le c\rho(m)\e^{1/2}(1+\e^{-1/2}|\al|)^{-1},\ \forall\al\in U,
\end{align}
for any $m\in\mathbb Z$, $|m|\le \e^{-\ga}$, and any interval $U\subset I$, $|U|=\e^{1/2}$. Therefore, if $|m|\le \e^{-\ga}$, one has
\be\label{main coef orig}\begin{split}
&\bigl(1+\e\phi_{\text{mx}}^{\prime\prime}\bigr)\TV(g_{m,\e},I)+\phi_{\text{mx}}^{\prime\prime}\int_{I}|g_{m,\e}(\al)|\dd\al\le c \rho(m)\e^{1/2}\ln(1/\e).
\end{split}
\ee
\end{lemma}

Comparing \eqref{g_e ass 1 v2} and \eqref{g_e ass 2 v2} with \eqref{g_e ass 1} and \eqref{g_e ass 2}, respectively, we see that $g_{m,\e}$ satisfy the assumptions of Lemma~\ref{lem:int1 est}.

From the statement of Lemma~\ref{lem:int1 est} it is clear that the intervals $I_l$ mentioned in the three-step procedure following \eqref{Jme sum A} should be carefully selected. Let us now construct these intervals.

Let $\alst\in(-\pi/2,\pi/2]$ satisfy $\phi^{\prime\prime}(\alst)=0$, i.e. $\vec\al_\star\cdot x_0=0$. Denote
\be\label{special l}
l_\star:=\lfloor\kappa \phi’(\alst)\rfloor,
\ee
see Figure~\ref{fig:sine}. For brevity, the dependence of $l_\star$ on $m$ is omitted from notation. Note that $\al_{l_\star}^*\not=\al_\star$ (i.e., $\{\kappa\phi’(\alst)\}>0$) for any $m\in\mathbb Z$, $m\not=0$, because $|\kappa\phi’(\al_\star)|=|m \kappa x_0|$ is irrational (see assumption (3) in Definition~\ref{def:gp}). Moreover, $\phi^{\prime\prime}(\alst)=0$ implies $\vec\al_\star\cdot x_0=0$, so $\alst\not=0$ by assumption (2) in Definition~\ref{def:gp}.

It is clear that $\al=\alst$ is either a local maximum or a local minimum of $\phi’$. 
Both cases are analogous, so we consider only one of them: $\al=\alst$ is a local maximum (as shown in  Figure~\ref{fig:sine}).

Define the set (see the red interval in Figure~\ref{fig:sine})
\be\label{exc intervals}\begin{split}
&I_{\star}=\{|\al|\le\pi/2: \kappa\phi’(\al)\ge l_\star\}.
\end{split}
\ee
Additionally, with each integer $l\in\mathcal L(m):=\kappa\phi’([-\pi/2,\pi/2])$ we associate the set:
\be\label{main intervals}\begin{split}
&I_l:=\{\al\in[-\pi/2,\pi/2]\setminus I_{\star}: |\kappa\phi’(\al)-l|\le 1/2\},
\end{split}
\ee
see the blue interval in Figure~\ref{fig:sine}. Each $I_l$ is the union of at most two intervals (on either side of $\alst$), and there are $O(|m|)$ of them. As is easily seen, each $I_l$ (or, each subinterval that makes up $I_l$) satisfies the assumptions of Lemma~\ref{lem:int1 est} (see \eqref{Omega_eps v2}). By construction,
\be\label{union}
\cup_{l\in\mathcal L(m)}I_l \cup I_{\star}=[-\pi/2,\pi/2].
\ee

The intervals $I_l$ are ‘regular’, in the sense that each of them satisfies the assumptions of Lemma~\ref{lem:int1 est}, so the corresponding integrals can be estimated using \eqref{int1 est}. The interval $I_{\star}$ is exceptional: \eqref{int1 est} does not apply, because $\phi_{\text{mn}}^{\prime\prime}(I_{\star})=0$ ($\phi^{\prime\prime}(\alst)=0$ and $\alst\in I_{\star}$). Thus, in addition to the regular intervals $I_l$, we have to consider the exceptional interval $I_{\star}$. Since Lemma~\ref{lem:int1 est} does not apply to it, estimation of its contribution requires special handling.

\begin{figure}[h]
{\centerline
{\epsfig{file=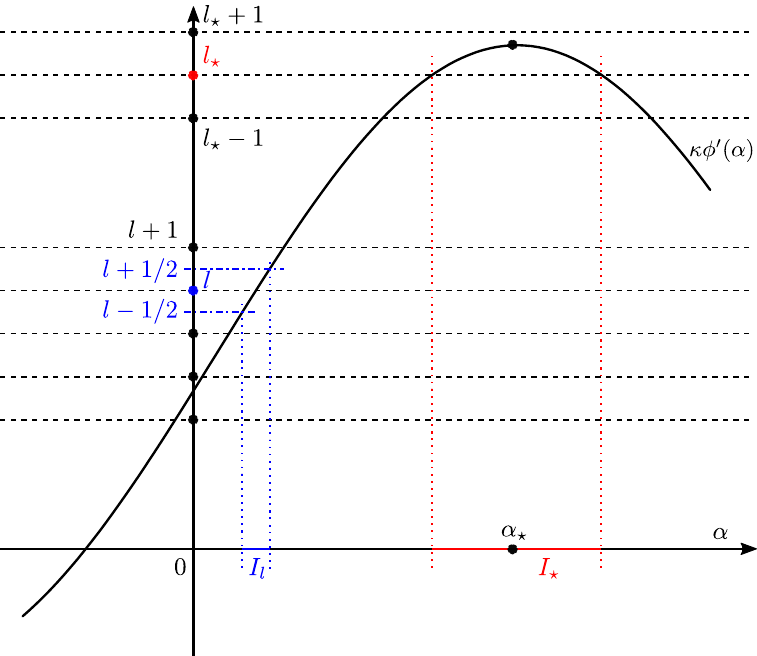, width=11.0cm}}}
\caption{The graph of $\kappa\phi’(\al)$ and construction of the intervals $I_l$ (blue) and $I_{\star}$ (red).}
\label{fig:sine}
\end{figure}

\begin{lemma}\label{lem:exc cases} Under the assumptions of Theorem~\ref{main-res} and assumptions~\ref{ass:S_A} one has
\be\begin{split}
\label{exc II}
\sum_{1\le |m|\le\e^{-\ga}}\biggl|\sum_{\al_k\in I_\star}g_{m,\e}(\al_k)e\biggl(\frac{\phi(\al_k)}\e\biggr)\biggr|
&=O(\e^{1/2}\ln(1/\e))\text{ if } \textcolor{black}{\bt>(\eta_\star/2)+1.}
\end{split}
\ee
\end{lemma}

Lemma~\ref{lem:sum1 est} and inequality \eqref{main coef orig} imply that the replacement of the sum \eqref{generic sum}, where $I=I_l$, with the corresponding integral \eqref{int jl*} leads to an error of magnitude $\sim \rho(m)\e^{1/2}\ln(1/\e)$. There are $O(|m|)$ values of $l$ for each $m$. Adding these error estimates for all $l$ and $m\not=0$ implies that the magnitude of the sum of the errors is $O(\e^{1/2}\ln(1/\e))$
as long as \textcolor{black}{$\bt>2$.} Hence, replacing the sums with the integrals is justified. Let the right side of \eqref{int1 est} be denoted $W_{l,m}(\e)$.

\begin{lemma}\label{lem:generic} Under the assumptions of Theorem~\ref{main-res} and assumptions~\ref{ass:S_A} one has
\be\label{sum-all}
\sum_{1\le |m|\le\e^{-\ga}}\sum_{l\in\mathcal L(m)} W_{l,m}(\e)=O(\e^{1/2}\ln(1/\e))\text{ if }  \textcolor{black}{\bt>\max(\eta_0+2,\eta_\star+1).}
\ee
\end{lemma}

Inequality \eqref{main coef orig} (with $m=0$) is the main ingredient in the proof of \eqref{extra terms II} (see appendix~\ref{ssec:ineq II}). 

\begin{lemma}\label{lem:ineq II} Equation \eqref{extra terms II} holds under the assumptions of Theorem~\ref{main-res} and assumptions~\ref{ass:S_A}. 
\end{lemma}

Combining Lemmas~\ref{lem:exc cases}, \ref{lem:generic} and \ref{lem:ineq II} completes the proof of  Lemma~\ref{lem:main-res-A}.

\section{Proof of \eqref{extra terms I} and \eqref{extra terms II} in case (B)}\label{sec:caseB}

At a high level, the proofs in cases (A) and (B) follow the same outline, see the five steps following \eqref{extra terms II}. By using a partition of unity, if necessary, we can make the following additional assumptions.

\begin{assumptions}\label{ass:S_B} 
$\hspace{1cm}$
\begin{enumerate}
\item $\s$ is sufficiently short and has nonzero curvature at every point;
\item $\s$ is parametrized as follows $[-a,a]\ni \theta\to y(\theta)\in\s$, where the interior unit normal to $\s$ at $y(\theta)$ is $\vec\theta=(\cos\theta,\sin\theta)$;
\item The line through $x_0$ and $y(0)$ is tangent to $\s$ at $y(0)$; and
\item $(y(0)-x_0)\cdot y’(0)>0$.
\end{enumerate}
\end{assumptions}

Assumptions~\ref{ass:S_B}(1,3) imply that no line through $x_0$, other than the line through $x_0$ and $y(0)$, is tangent to $\s$. Assumption~\ref{ass:S_B}(4) means that $x_0$ is to the left of $y(0)$ as shown in  Figure~\ref{fig:rem A_fn}. The case when $x_0$ is to the right of $y(0)$ is completely analogous and is not considered here.

Define $\CA(\theta)$ as the angle $\al\in (-\pi/2,\pi/2)$ such that $\vec\al(\CA(\theta))\cdot (y(\theta)-x_0)\equiv0$, $|\theta|\le a$. In particular, $\CA(0)=0$ (see Figure~\ref{fig:rem A_fn}). The function $\CA(\theta)$ is no longer injective. It is injective if its domain is restricted to $[0,a]$ or $[-a,0]$.

\begin{figure}[h]
{\centerline
{\epsfig{file=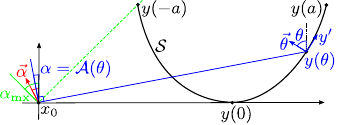, width=8.4cm}}}
\caption{Illustration of case (B).}
\label{fig:rem A_fn}
\end{figure}

Since $x_0\not\in\s$, we use the function $R$ in its original form $R(\theta,\al)=\vec\al\cdot (y(\theta)-x_0)$ (cf. \eqref{A simpl orig}). 
\begin{lemma}\label{lem:aux props rem} Under the assumptions of Theorem~\ref{main-res} and assumptions~\ref{ass:S_B} one has 
\be\label{A props rem}\begin{split}
&\CA(\theta)\asymp\theta^2,\ \CA'(\theta)\asymp\theta,\ |\theta|\le a;\\
&|(\CA^{-1})^{(j)}(t)|=O(|t|^{(1/2)-j}),\ j=0,1,2,\ t\to0.
\end{split}
\ee
Here $\CA^{-1}$ is computed by restricting its range to $[0,a]$ or $[-a,0]$. Also,
\be\label{r props rem}
R(\theta,\al)=b(\theta,\al)(\CA(\theta)-\al),\ (\theta,\al)\in\CD.
\ee
Here $b\in C^1(\CD)$ and $|b(\theta,\al)|\asymp1$, $(\theta,\al)\in\CD$.
\end{lemma}

All lemmas in this section are proven in appendix~\ref{sec:caseB proofs}. 
Let $\al_{\text{mx}}$ be such that $[0,\al_{\text{mx}}]=\CA([-a,a])$. In other words, $[0,\al_{\text{mx}}]$ is the set of all $|\al|\le \pi/2$ such that the lines $\{y\in\br^2:\,(y-x_0)\cdot\vec\al=0\}$ intersect $\s$ (see Figure~\ref{fig:rem A_fn}). By shrinking $\s$ even more, if necessary, we can assume that $\al_{\text{mx}}>0$ is as small as we like. Similarly to \eqref{A simpl orig} and \eqref{Am small g}, define
\be\label{Am small g rem}\begin{split}
&g_{m,\e}^{(2)}(\al):=e\left(-m \vec\al\cdot \check x\right)\\
&\times
\int_{|\CA(\theta)-\al| \le (\e\al)^{1/2}}
\int_0^{H_0(\e^{-1/2}\theta)}\tilde\psi_m\left(\e^{-1}R(\theta,\al)+h(\hat t,\theta,\al)\right)
F(\theta,\e \hat t)\dd\hat t\dd\theta,
\end{split}
\ee
where $h$ and $F$ are the same as in \eqref{A-simpl} and \eqref{A simpl orig} and $\al\in\Omega_\e:=[4\e,\al_{\text{mx}}]$.

Denote
\be\label{2 Xi sets}\begin{split}
&\Xi:=\{(m,\al_k):\ m\in\mathbb Z, |m|\le \e^{-\ga},\al_k\in\Omega_\e\},\\
&\Xi^c:=\{(m,\al_k):\ m\in\mathbb Z, |\al_k|\le \pi/2\}\setminus \Xi.
\end{split}
\ee

\begin{lemma}\label{lem:rem A-to-g2} Under the assumptions of Theorem~\ref{main-res} and assumptions~\ref{ass:S_B}, one has
\be\label{rem sums A-g22}\begin{split}
&\sum_{(m,\al_k)\in\Xi^c} |\e A_m(\al_k,\e)|=O(\e^{1/2}\ln(1/\e))\text{ and}\\
&\sum_{(m,\al_k)\in\Xi} |\e A_m(\al_k,\e)-g_{m,\e}^{(2)}(\al_k)|=O(\e^{1/2}\ln(1/\e))
\end{split}
\ee
when \textcolor{black}{$\bt>1$.}
\end{lemma}

By Lemma~\ref{lem:rem A-to-g2}, to prove \eqref{extra terms I} it remains to estimate the double sum on the right in \eqref{rem sums A-g22} (without the terms with $m=0$). 

\begin{lemma}\label{lem:case B last} Under the assumptions of Theorem~\ref{main-res} and assumptions~\ref{ass:S_B}, one has
\be\label{rem sum last}\begin{split}
\sum_{1\le|m|\le \e^{-\ga}}\sum_{\al_k\in\Omega_\e} g_{m,\e}^{(2)}(\al_k)e(\phi(\al_k)/\e)=O(\e^{1/2}\ln(1/\e))\text{ if }\textcolor{black}{\bt>\eta_0+1.}
\end{split}
\ee
\end{lemma}

This completes the proof of \eqref{extra terms I} in case (B). 

\begin{lemma}\label{lem:ineq II B} Equation \eqref{extra terms II} holds under the assumptions of Theorem~\ref{main-res} and assumptions~\ref{ass:S_B}.
\end{lemma}

Thus, we have proved Lemma~\ref{lem:main-res-BC} in case (B).

\section{Proof of Lemma~\ref{lem:main-res-BC} in case (C)}\label{sec:caseC}

Let $V$ be a small neighborhood of $x_0$. By using a partition of unity, if necessary, we can make the following additional assumptions in case (C).

\begin{assumptions}\label{ass:S_C}
$\hspace{1cm}$
\begin{enumerate}
\item The parametrization of $\s$ is in terms of some regular map $[-a,a]\ni u\to y(u)\in C^4$;
\item $\s$ is sufficiently short and $|(y(u)-x)^\perp\cdot y’(u)|\ge c>0$, $|u|\le a$, $x\in V$; and
\item There exists $0<\de<a$ such that $\fpe(y(u)+t\vec\theta(u))=0$ for all $a-\de\le |u|\le a$, $t\in\br$ and $\e>0$ sufficiently small.
\end{enumerate}
\end{assumptions}

\begin{figure}[h]
{\centerline
{\epsfig{file=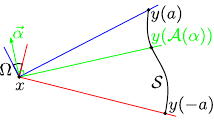, width=6.0cm}}}
\caption{Illustration of case (C).}
\label{fig:rem2 A_fn}
\end{figure}

All the estimates below are uniform with respect to $x\in V$, so the $x$-dependence of various quantities is omitted from notation. Pick any $x\in V$ and assume that it is fixed throughout this section. Let $\Omega$ be the set of all $|\al|\le\pi/2$ such that the lines $\{y\in\br^2:\,(y-x)\cdot\vec\al=0\}$ intersect $\s$. Let $u=\CA(\al)$, $\al\in\Omega$, be determined by solving $(y(u)-x)\cdot\vec\al=0$, see Figure~\ref{fig:rem2 A_fn}. By  assumptions~\ref{ass:S_C}, the solution $u=\CA(\al)$ is unique for each $\al\in\Omega$.

By the definition of case (C), the intersection between any line through $x$ and $\s$ is transverse (up to the endpoints of $\s$). Hence 
$|\CA'(\al)|=|y(\CA(\al))-x|/|\vec\al\cdot y'(\CA(\al))|$ and 
\be\label{theta deriv min}
0<\min_{\al\in\Omega}|\vec\al\cdot y'(\CA(\al))|,\ 
0<\min_{\al\in\Omega}|\CA'(\al)|\le \max_{\al\in\Omega}|\CA'(\al)|<\infty.
\ee

Define similarly to \eqref{Am small g rem}, 
\be\label{rem2 gme2}\begin{split}
g_{m,\e}^{(1)}(\al):=&\int_{|u-\CA(\al)|\le\e^{1/2}}\int_0^{H_0(\e^{-1/2}u)}\tilde\psi_m\left(\e^{-1}R(u,\al)+h(\hat t,u,\al)\right)\\
&\hspace{3cm}\times F(u,\e \hat t)\dd\hat t\dd u,\ \al\in\Omega,\\
R(u,\al):=&\vec\al\cdot (y(u)-x),\ F(u,t):=\fpe(y(u)+t\vec\theta(u)) |\pa y/\pa(u,t)|,\\ 
h(\hat t,u,\al):=&\hat t\cos(\theta(u)-\al),
\end{split}
\ee
where $\vec\theta(u)$ is a continuous unit vector function such that $\vec\theta(u)\cdot y’(u)\equiv0$.
Since $\check x$ is absorbed by $x$, the term $\al\cdot\check x$ is not a part of the function $h$. 

\begin{lemma}\label{lem:rem2 A-to-g2} Under the assumptions of Theorem~\ref{main-res} and assumptions~\ref{ass:S_C}, one has
\be\label{rem2 sums A-g22}\begin{split}
\sum_{m\in\BZ}& \sum_{|\al_k|\le\pi/2,\al_k\not\in\Omega} |\e A_m(\al_k,\e)|
= O(\e)\quad\text{and}\\
\sum_{m\in\BZ}&\sum_{\al_k\in\Omega} |\e A_m(\al_k,\e)-g_{m,\e}^{(1)}(\al_k)|
= O(\e^{1/2})\quad\text{if}\quad\textcolor{black}{\bt>1.}
\end{split}
\ee
Moreover, for any $m\in\mathbb Z$ one has
\be\label{rem2 g2 bnd}
|g_{m,\e}^{(1)}(\al)|\le c\rho(m)\e,\ \al\in\Omega,
\ee
and
\be\label{rem2 TV bnd}
\TV(g_{m,\e}^{(1)},U)\le c\rho(m)\e \text{ for any interval } U\subset\Omega,\ |U|=O(\e^{1/2}). 
\ee
\end{lemma}


\begin{lemma}\label{lem:rem2 sum g2} Under the assumptions of Theorem~\ref{main-res} and assumptions~\ref{ass:S_C}, one has
\be\label{rem2 sum g2}\begin{split}
\sum_{m\in\BZ,m\not=0}\sum_{\al_k\in \Omega} g_{m,\e}^{(1)}(\al_k)e(\phi(\al_k)/\e)=O(\e^{1/2}\ln(1/\e))
\end{split}
\ee
if \textcolor{black}{$\bt>\max(2,(\eta_\star/2)+1)$.}
\end{lemma}

Combining \eqref{rem2 sums A-g22} and \eqref{rem2 sum g2} proves \eqref{extra terms I}. 

\begin{lemma}\label{lem:rem2 sum g3} Equation \eqref{extra terms II} holds under the assumptions of Theorem~\ref{main-res} and assumptions~\ref{ass:S_C}. 
\end{lemma}

The proofs of lemmas~\ref{lem:rem2 A-to-g2}, \ref{lem:rem2 sum g2} and \ref{lem:rem2 sum g3} are in  appendix~\ref{sec:caseC proofs}. Combining these three lemmas proves Lemma~\ref{lem:main-res-BC} in case (C).

\appendix
\section{Proof of the Lemmas~\ref{lem:sum1 est} and \ref{lem:int1 est}}\label{sec:prf sum_int}

\subsection{Proof of Lemma~\ref{lem:sum1 est}}\label{ssec:prf sum1 est}

Clearly, the sum over $\al_k$ on the left in \eqref{sum est} does not change if $\phi$ is replaced by $\fks$. Rewrite the top line in \eqref{sum est} as follows
\be\label{diff v1}\begin{split}
R(\e):=\frac1{\Delta\al}\sum_{\al_k\in\oks} \int_{|\al-\al_k|\le \Delta\al/2}\bigl[&g(\al)h_l(\al) e({\fks(\al)}/\e)
-g(\al_k) e(\fks(\al_k)/\e)\bigr]\dd\al.
\end{split}
\ee
The expression in brackets can be written as follows
\be\label{diff v3}\begin{split}
&[g(\al)-g(\al_k)]h_l(\al) e(\fks(\al)/\e)+g(\al_k)\bigl[h_l(\al) e(\fks(\al)/\e)-e(\fks(\al_k)/\e)\bigr].
\end{split}
\ee

By construction, $1\le h_l(\al)\le \pi/2$, $\al\in\oks$. Therefore,
\be\label{diff v2}\begin{split}
|R(\e)|\le& \frac\pi2R^{(1)}(\e)+ \sum_{\al_k\in\oks}|g(\al_k)| |R_k^{(2)}(\e)|,\\
R^{(1)}(\e):=&\frac1{\Delta\al}\sum_{\al_k\in\oks} \int_{|\al-\al_k|\le \Delta\al/2}\bigl|g(\al)-g(\al_k)\bigr|\dd \al,\\ 
R_k^{(2)}(\e):=&\frac1{\Delta\al}\int_{|\al-\al_k|\le \Delta\al/2}\bigl[h_l(\al) e([\fks(\al)-\fks(\al_k)]/\e)-1\bigr]\dd\al.
\end{split}
\ee
The term $R^{(1)}(\e)$ is estimated as follows: 
\be\label{sum c2}\begin{split}
& R^{(1)}(\e) \le 
\sum_{\al_k\in\oks} |g(\bar\al_k)-g(\al_k)| \le \TV(g,I),\\
& \bar \al_k:=\text{argmax}_{\al:|\al-\al_k|\le \Delta\al/2}|g(\al)-g(\al_k)|.
\end{split}
\ee

Consider next the second term on the first line in \eqref{diff v2}. Changing variables $t=(\al-\al_k)/\Delta\al$ and expanding $\fks$ in the exponent gives
\be\label{Jn est1}\begin{split}
R_k^{(2)}(\e)=&\int_{|t|\le 1/2}h_l(\al_k+t\Delta\al) e\bigl(\kappa  \fks’(\al_k)t+O(\e \phi_{\text{mx}}^{\prime\prime})\bigr)\dd t-1\\
=& \biggl[h_l(\al_k) \int_{|t|\le 1/2}e(\kappa  \fks’(\al_k)t)\dd t-1\biggr] +O(\e \phi_{\text{mx}}^{\prime\prime})\\
&+\int_{{|t|\le 1/2}}(h_l(\al_k+t\Delta\al)-h_l(\al_k))O(1)\dd t.
\end{split}
\ee
Since the expression in brackets is zero due to the choice of $h_l$, simple estimates show
$|R_k^{(2)}(\e)|\le  c\e \phi_{\text{mx}}^{\prime\prime}$.
Here we have used that $\max_{\al\in\oks}|\kappa\fks’(\al)|\le 1/2$ (cf. \eqref{Omega_eps} and \eqref{h-fn}) and
\be\label{del h est1}\begin{split}
\max_{|t|\le 1/2}|h_l(\al_k+t\Delta\al)-h_l(\al_k)|\le & c\e\max_{\al\in\oks}|h_l’(\al)|
\le c\e \phi_{\text{mx}}^{\prime\prime}.
\end{split}
\ee
Hence
\be\label{diff v2 bnd}\begin{split}
\sum_{\al_k\in\oks}|g(\al_k)| |R_k^{(2)}(\e)|
&\le c \phi_{\text{mx}}^{\prime\prime} \e \sum_{\al_k\in\oks}|g(\al_k)|\\
&\le c \phi_{\text{mx}}^{\prime\prime} \biggl(\int_I |g(\al)|\dd\al+\e\TV(|g|,I)\biggr).
\end{split}
\ee
Clearly, $\TV(|g|,I)\le \TV(g,I)$. Combining the results proves the lemma.

\subsection{Proof of Lemma~\ref{lem:int1 est}}\label{lem:int1 est prf}
Without loss of generality, we may assume that $\al_l^*$ is an endpoint of $I_l$. Otherwise, we can write  $I_l=[\al_l^*-\de’,\al_l^*]\cup[\al_l^*,\al_l^*+\de]$ for some $\de,\de’>0$ and consider each of the two intervals separately. Here we estimate the contribution of the second interval. The first interval can be considered completely analogously. Note that $\al_l^*$ can be negative and it may happen that $\al_l^*\le 0\le \al_l^*+\de$. Consider two subintervals 
\be\label{aux ints}
U_1:=[\al_l^*,\al_l^*+\e^{1/2}],\ U_2:=[\al_l^*+\e^{1/2},\al_l^*+\de],\ U_1\cup U_2=[\al_l^*,\al_l^*+\de],
\ee
and denote 
\be\label{Jn def first}
J_n:=\frac1{\Delta\al}\int_{U_n} (g_{\e} h_l)(\al) e\biggl(\frac{\phi_l(\al)}\e\biggr)\dd \al,\ n=1,2.
\ee
Using \eqref{g_e ass 1} and that $\al_l^*+\e^{1/2}\asymp\al_l^*$ (which follows from \eqref{ang0 bnd} below and $\bt>\eta_0+2$), direct estimation implies
\be\label{sum est 4}
J_1\le c\rho(m)/|\tilde p|,\quad \tilde p:=\e^{-1/2}\al_l^*.
\ee
A slightly more accurate estimate, which decays faster as $\e\to 0$, can be obtained by integrating by parts (differently than in \eqref{sum est 1 v2} below). However, we do not need it here.

Integration by parts yields
\be\label{sum est 1 v2}\begin{split}
&J_2
=\pi[J_{2,1}-J_{2,2}],\
J_{2,1}:=\left. \frac{g_\e(\al)}{\sin(\pi\kappa  \phi_l’(\al))}\right|_{\al_l^*+\e^{1/2}}^{\al_l^*+\de},\\
&J_{2,2}:=\int_{U_2}\pa_\al\frac{g_\e(\al)}{\sin(\pi\kappa  \phi_l’(\al))}e\biggl(\frac{\phi_l(\al)}\e\biggr)\dd\al,
\end{split}
\ee
where we have used the definition of $h_l$ in \eqref{h-fn}. Obviously, 
\be\label{sin bnd}
|\phi_l’(\al)|\ge c\phi_{\text{mn}}^{\prime\prime}|\al-\al_l^*|,\ \al\in[\al_l^*,\al_l^*+\de].
\ee
Recall that here $\phi_{\text{mn}}^{\prime\prime}=\phi_{\text{mn}}^{\prime\prime}(I_l)$. By the definition of the interval $I_l$, $\kappa\phi_l’(I_l)\subset[-1/2,1/2]$, so $\sin(\pi\kappa  \phi_l’(\al))\asymp \kappa  \phi_l’(\al)$ on $I_l$. 
Then, by \eqref{g_e ass 1} and \eqref{sin bnd}, 
\be\label{int term v2}\begin{split}
|J_{2,1}|\le c\rho(m)\e^{1/2}\biggl[\frac{1}{1+\e^{-1/2}|\al_l^*+\e^{1/2}|}\frac1{\phi_{\text{mn}}^{\prime\prime}\e^{1/2}}&\\
+\frac{1}{1+\e^{-1/2}|\al_l^*+\de|}\frac1{\phi_{\text{mn}}^{\prime\prime}\de}&\biggr]
\le c \frac{\rho(m)}{\phi_{\text{mn}}^{\prime\prime}}\frac1{|\tilde p|}.
\end{split}
\ee
We have again used here that $\al_l^*+\e^{1/2}\asymp\al_l^*$ and
\be
|\al_l^*|\le (1+\e^{-1/2}|\al_l^*+\de|)\de\quad \text{if}\quad \de>\e^{1/2}.
\ee
Otherwise, if $\de\le\e^{1/2}$, then $J_2=0$ and only $J_1$ remains.
Further,
\be\label{sum est 1 v3}\begin{split}
|J_{2,2}|\le &\int_{U_2}\biggl|\pa_\al\frac{g_\e(\al)}{\sin(\pi\kappa  \phi_l’(\al))}\biggr|\dd\al\\
\le & c\sum_{n=2}^N\biggl[\max_{\al\in U_{2,n}}\biggl(\frac1{|\phi_l’(\al)|}\biggr)\TV(g_\e,U_{2,n})+
\int_{U_2} |g_\e(\al)|\biggl|\biggl(\frac1{\phi_l’(\al)}\biggr)’\biggr|\dd\al\biggr],\\
U_{2,n}:=&\al_l^*+\e^{1/2}[n-1,n],\ n=2,\dots,N-1,\\ 
U_{2,N}:=&\al_l^*+[\e^{1/2}(N-1),\de],
\end{split}
\ee
where $N=\lceil\de/\e^{1/2}\rceil$. By \eqref{g_e ass 1}, \eqref{g_e ass 2} and \eqref{sin bnd}, 
\be\label{integral br}\begin{split}
|J_{2,2}|\le &c\rho(m)\biggl[\frac{1}{\phi_{\text{mn}}^{\prime\prime}}J_{2,2}^a+\frac{\phi_{\text{mx}}^{\prime\prime}}{(\phi_{\text{mn}}^{\prime\prime})^2}J_{2,2}^b \biggr],\\
J_{2,2}^a:=&\sum_{n=2}^N\frac{1}{\e^{1/2}(n-1)}\frac{\e^{1/2}}{1+\e^{-1/2}|\al_l^*+\e^{1/2}(n-1)|},\\
J_{2,2}^b:=&\int_{U_2}\frac{\e^{1/2}}{1+\e^{-1/2}|\al|}\frac{1}{(\al-\al_l^*)^2}\dd\al,
\end{split}
\ee
where $\phi_{\text{mx}}^{\prime\prime}=\phi_{\text{mx}}^{\prime\prime}(I_l)$ and $\phi_{\text{mn}}^{\prime\prime}=\phi_{\text{mn}}^{\prime\prime}(I_l)$. After some tedious calculations (see section~\ref{ssec:J22ab}):
\be\label{J22ab}
J_{2,2}^a \le c\ln|\tilde p|/|\tilde p|,\ J_{2,2}^b \le c/|\tilde p|.
\ee
Recall that $|\tilde p|>1$. Combining \eqref{sum est 4}, \eqref{sum est 1 v2}, \eqref{int term v2}, \eqref{integral br} and \eqref{J22ab} proves the lemma.

\section{Proof of Lemma~\ref{lem:apply lemma 1}}\label{sec:prf appl lem1}

Let $\Omega$ be a small neighborhood of $\CA([-a,a])$. First, we state and prove a useful lemma.

\begin{lemma}\label{aux lem sth}
Consider the function
\be\label{s_vs_theta}
s(\theta,\al):=\frac{R(\theta,\al)}\theta=\vec\al\cdot\frac{y(\theta)-y(0)}{\theta},\
(\theta,\al)\in\CD.
\ee
Under assumptions~\ref{ass:S_A} one has:
\begin{enumerate}
\item $s(\theta,\al)\asymp\CA(\theta)-\al$, $|\theta|\le a$, $\al\in\Omega$;
\item $s\in C^1(\CD)$, $\pa_\theta s(\theta,\al)\asymp 1$, and $\pa_\al s(\theta,\al)\asymp -1$, $|\theta|\le a$, $\al\in\Omega$;
\item The equation $s=s(\theta,\al)$ determines a function $\theta(s,\al)$ for all $|s|$ sufficiently small and $\al\in\Omega$; and
\item $\pa_s\theta(s,\al),\pa_\al\theta(s,\al)\asymp 1$ and $|\pa_\al\pa_s \theta(s,\al)|\le c$ when $|s|$ sufficiently small and $\al\in\Omega$.
\end{enumerate}
\end{lemma}

\begin{proof} Define $z(\theta):=(y(\theta)-y(0))/\theta$. By assumption~\ref{ass:f}(1), $z(\theta)$ is a nonzero, $C^3([-a,a])$ vector function satisfying
\be\label{zth props}
z(0)=y’(0),\ z’(0)=y^{\prime\prime}(0)/2.
\ee
As is easily seen,
\be\label{s_vs_theta alt}
s(\theta,\al)=|z(\theta)|\frac{\sin(\CA(\theta)-\al)}{\CA(\theta)-\al}(\CA(\theta)-\al),\ 
|\theta|\le a,\al\in\Omega.
\ee
This proves statement (1) and that $s\in C^1(\CD)$. By \eqref{s_vs_theta}, \eqref{zth props} and assumption~\ref{ass:S_A}(2), 
\be\begin{split}
&\pa_\theta s(\theta,\al)|_{\theta=0}=\vec\al\cdot y^{\prime\prime}(0)/2=|y’(0)|/2>0,\\
&\pa_\al s(0,\al)=\vec\al^\perp\cdot y^{\prime}(0)=-|y’(0)|<0.
\end{split}
\ee
Combined with assumption~\ref{ass:S_A}(1) this proves the rest of statement (2), statement (3) and that $\pa_s\theta(s,\al)\asymp 1$. The rest of statement (4) is proven by differentiating the identity $s\equiv s(\theta(s,\al),\al)$ with respect to $\al$ and then $s$ and using statement (2).
\end{proof}

We now return to the proof of Lemma~\ref{lem:apply lemma 1}. 
\subsection{Estimate $g_{m,\e}(\al)$}
From Lemmas~\ref{lem:aux props}, \ref{lem:first int} and \eqref{Am small g},
\begin{align}\label{eA 1st}
&|g_{m,\e}(\al)|\le c\rho(m)\e^{1/2}(1+\e^{-1/2}|\al|)^{-1},\ \al\in I,\\
\label{eA 2nd}
&\int_{I}|g_{m,\e}(\al)|\dd\al\le c\rho(m)\int_{I}\frac{\e^{1/2}}{1+\e^{-1/2}|\al|}\dd\al\le c\rho(m)\e\ln(1/\e).
\end{align}
In addition, from \eqref{recon-ker-v2} and the differentiability of $\tilde\psi_m$ it follows that $g_{m,\e}\in C^1([-\pi/2,\pi/2])$.

\subsection{Split $g_{m,\e}=g_{m,\e}^{(1)}+g_{m,\e}^{(2)}$ and estimate $\TV(g_{m,\e}^{(1)},I)$}\label{ssec:step2}
Define two sets 
\be\label{three Theta sets}\begin{split}
&\Xi_1(\al,\e):=\{|\theta|\le a:\,|s(\theta,\al)|\le\e^{1/2}\},\
\Xi_2(\al,\e):=[-a,a]\setminus\Xi_1(\al,\e),
\end{split}
\ee
and the associated functions (cf. \eqref{A simpl orig} and \eqref{Am small g})
\be\label{three As}
\begin{split}
g_{m,\e}^{(n)}(\al):=e\left(-m \vec\al\cdot \check x\right)\int_{\Xi_n(\al,\e)}\int_0^{H_0(\e^{-1/2}\theta)}&\tilde\psi_m\left(\e^{-1}R(\theta,\al)+h(\hat t,\theta,\al)\right)\\
&\times F(\theta,\e \hat t)\dd\hat t\dd\theta,\ n=1,2.
\end{split}
\ee
Clearly, $g_{m,\e}\equiv g_{m,\e}^{(1)}+g_{m,\e}^{(2)}$. If $\Xi_1(\al,\e)=\varnothing$ for some $\al$, we set $g_{m,\e}^{(1)}(\al)=0$. Therefore, in this subsection we consider only $\al$ for which $\Xi_1(\al,\e)\not=\varnothing$.

When $\e>0$ is small and $\theta\in\Xi_1(\al,\e)$, $s(\theta,\al)$ is close to zero and $\al$ is close to $\CA(\theta)$, i.e. $\al\in\Omega$. This means that $\Xi_1(\al,\e)\not=\varnothing$ implies $\al\in\Omega$ and  Lemma~\ref{aux lem sth} applies. Also, $\TV(g_{m,\e}^{(1)},I)=\TV(g_{m,\e}^{(1)},\Omega)$.

Next we estimate $\TV(g_{m,\e}^{(1)},\Omega)$. Changing variable $\theta\to s=R(\theta,\al)/\theta$, which is possible by statement (4) of Lemma~\ref{aux lem sth}, gives
\be\label{small g v2}\begin{split}
g_{m,\e}^{(1)}(\al)=e\left(-m \vec\al\cdot \check x\right)&\int_{|s|\le\e^{1/2}}\int_0^{H_0(\e^{-1/2}\theta)}  \tilde\psi_m\biggl(\frac{\theta s}\e +h(\hat t,\theta,\al)\biggr),\\
&\times F(\theta,\e \hat t)\dd\hat t\,\pa_s\theta(s,\al)\dd s,\quad \theta=\theta(s,\al).
\end{split}
\ee
Recall that $\pa_s\theta(s,\al)>0$ by statement (3) of Lemma~\ref{aux lem sth}.
Consider an interval $U\subset \Omega$ of length $\e^{1/2}$, which may depend on $\e$. Represent $g_{m,\e}^{(1)}$ in the form
\be\label{small g v1 alt}\begin{split}
g_{m,\e}^{(1)}(\al)=&\int_{|s|\le\e^{1/2}}\int_0^{\mu(\al,s)} G(\al,s,\hat t,\e)\dd\hat t \dd s,\\
G(\al,s,\hat t,\e):=&e\left(-m \vec\al\cdot \check x\right)\tilde\psi_m\biggl(\frac{\theta s}\e +h(\hat t,\theta,\al)\biggr)
F(\theta,\e \hat t)\,\pa_s\theta(s,\al),\\
\mu(\al,s):=&H_0(\e^{-1/2}\theta),\quad \theta=\theta(s,\al),\ \al\in U.
\end{split}
\ee
The dependence of some functions on $\e$ is suppressed for simplicity. Clearly,
\be\label{g findif}\begin{split}
&g_{m,\e}^{(1)}(\al_2’)-g_{m,\e}^{(1)}(\al_1’)=J_1(\al_1’,\al_2’)+J_2(\al_1’,\al_2’),\ \al_1’,\al_2’\in U,\\
&J_1(\al_1’,\al_2’):=\int_{|s|\le\e^{1/2}}\int_{\mu(\al_1’,s)}^{\mu(\al_2’,s)} G(\al_2’,s,\hat t,\e)\dd\hat t \dd s,\\
&J_2(\al_1’,\al_2’):=\int_{|s|\le\e^{1/2}}\int_0^{\mu(\al_1’,s)} (G(\al_2’,s,\hat t,\e)-G(\al_1’,s,\hat t,\e))\dd\hat t \dd s.
\end{split}
\ee
We use primes in $\al_k’$, $k=1,2$ (and for other values of $k$ below), in order not to confuse these auxiliary angles with the data points $\al_k$ in \eqref{data_eps}. Then
\be\label{J1 est}\begin{split}
&|J_1(\al_1’,\al_2’)|\le \int_{|s|\le\e^{1/2}}\max_{|\hat t|\le c}|G(\al_2’,s,\hat t,\e)| |H_0(\tilde\theta_2’)-H_0(\tilde\theta_1’)|\dd s,\\ 
&\tilde\theta_k:=\e^{-1/2}\theta(s,\al_k’),\ k=1,2,\ 
|\tilde\theta_2’-\tilde\theta_1’|\le c\e^{-1/2}|\al_2’-\al_1’|,
\end{split}
\ee
and
\be\label{J2 est}\begin{split}
|J_2(\al_1’,\al_2’)|&\le \int_{|s|\le\e^{1/2}}\int_0^{\mu(\al_1’,s)}\int_{\al_1’}^{\al_2’} |\pa_\al G(\al,s,\hat t,\e)|\dd\al\dd\hat t \dd s\\
&\le c\int_{|s|\le\e^{1/2}}\int_{\al_1’}^{\al_2’} \max_{|\hat t|\le c}|\pa_\al G(\al,s,\hat t,\e)|\dd\al \dd s.
\end{split}
\ee
Pick any $K>0$ and consider a partition of $U$ 
\be\label{partition}
\al_0’<\al_1’<\cdots<\al_K’,\ [\al_0’,\al_K’]=U.
\ee
Set $\tilde U:=\e^{-1/2}\theta(s,U)$, apply \eqref{J1 est} to each pair $\al_{k-1}’,\al_k’$ and add the results:
\be\label{H0 rel est}\begin{split}
\sum_{k=1}^K |J_1(\al_{k-1}’,\al_k’)|&\le
c \rho(m)\int_{|s|\le\e^{1/2}}\sum_{k=1}^K \frac{|H_0(\tilde\theta_k’)-H_0(\tilde\theta_{k-1}’)|}{1+(\theta(s,\al)s/\e)^2}\dd s\\
&\le
c \rho(m)\int_{|s|\le\e^{1/2}}\frac{\TV(H_0,\tilde U)}{1+(\theta(s,\al)s/\e)^2}\dd s,
\end{split}
\ee
where $\al\in U$ is any point. Indeed, from \eqref{four-coef-bnd} and the fact that $h(\hat t,\theta,\al)$ is uniformly bounded, we have
\be\label{aux ge bnd}\begin{split}
&\biggl|\tilde\psi_m\biggl(\frac{\theta s}\e +h(\hat t,\theta,\al)\biggr)\biggr|
\le c\frac{\rho(m)}{1+(\theta s/\e)^2},\\ 
&|\hat t|\le c,\ \theta=\theta(s,\al),\ |s|\le\e^{1/2},\ \al\in U.
\end{split}
\ee
The remaining factors that make up $G$ (see \eqref{small g v1 alt}) are all uniformly bounded. By  Lemma~\ref{lem:aux props} and statement (1) of Lemma~\ref{aux lem sth}
it immediately follows that 
\be\begin{split}
1+(\theta(s,\al) s/\e)^2\asymp & 1+((s+\al) s/\e)^2\asymp 1+((s+\al’) s/\e)^2,\\ 
|s|\le&\e^{1/2},\ \al,\al’\in U,\ |U|=\e^{1/2}.
\end{split}
\ee

By construction, $|\tilde U|\asymp 1$ uniformly in $s$. 
Hence, by assumption~\ref{ass:H0}(2), $\TV(H_0,\tilde U)<c$ and, by Lemma~\ref{lem:first int} (where we replace $\theta$ with $s$ and $\al$ with $-\al$),
\be\label{H0 rel est st2}\begin{split}
\sum_{k=1}^K |J_1(\al_{k-1}’,\al_k’)|&\le
c \rho(m)\int_{|s|\le\e^{1/2}}[1+((s+\al)s/\e)^2]^{-1}\dd s\\
&\le c \rho(m)\e^{1/2}(1+\e^{-1/2}|\al|)^{-1},\ \forall\al\in U.
\end{split}
\ee

Consider now \eqref{J2 est}. Adding all the $J_2$ terms analogously to \eqref{H0 rel est} gives
\be\label{sum J2}\begin{split}
\sum_{k=1}^K |J_2(\al_{k-1}’,\al_k’)|&\le
c\int_U\int_{|s|\le\e^{1/2}}\max_{|\hat t|\le c}|\pa_\al G(\al,s,\hat t,\e)|\dd s\dd\al .
\end{split}
\ee
Computing the derivative $\pa_\al G$ gives the sum of several terms, most of which are estimated as in \eqref{eA 1st}. Exceptional cases arise when $\theta s/\e$ and the exponential factor are differentiated with respect to $\al$ (see \eqref{small g v1 alt}). 

When the exponential factor is differentiated, an extra factor $m$ appears. By Lemma~\ref{lem:first int} the resulting integral 
is bounded by
\be\label{exp der bnd}
c\rho(m)|m|\e^{1/2}(1+\e^{-1/2}|\alpha|)^{-1}.
\ee
When $\theta(s,\al) s/\e$ is differentiated, an extra factor $s/\e$ appears in the integrand. Due to statement (4) of Lemma~\ref{aux lem sth}, the corresponding model integral satisfies (see subsection~\ref{ssec: two sep ints} for the proof)
\be\label{model int add 1}\begin{split}
\int_{|s|\le\e^{1/2}}\frac{|s|/\e}{1+(s(\al+s)/\e)^2}\dd s=O(\tal^{-2}\ln|\tal|),\ \tal:=\e^{-1/2}\al\to\infty.
\end{split}
\ee

Combine \eqref{H0 rel est st2}--\eqref{model int add 1} to estimate the total variation
\be\label{var est (1)}\begin{split}
&\TV(g_{m,\e}^{(1)},U)\\
&\le c\rho(m)\biggl[\e^{1/2}\int_{\e^{-1/2}U}\biggl(\frac{|m|\e^{1/2}}{1+|\tal|}+\frac{\ln(2+|\tal|)}{1+\tal^{2}}\biggr)\dd\tal+\frac{\e^{1/2}}{1+\e^{-1/2}|\al|}\biggr]\\
&\le c\rho(m)\e^{1/2}(1+\e^{-1/2}|\al|)^{-1}\ \forall\al\in U.
\end{split}
\ee
Here we have used that (a) $|U|=\e^{1/2}$ implies $1+\e^{-1/2}|\al|\asymp 1+\e^{-1/2}|\al’|$ for any $\al,\al’\in U$ and (b) $|m|\le\e^{-\ga}$ implies $|m|\e^{1/2}\le c$. An extra factor $\e^{1/2}$ in front of the integral in \eqref{var est (1)} appears because the integral to evaluate $\TV(g_{m,\e}^{(1)},U)$ is with respect to $\al$ rather than $\tal=\e^{-1/2}\al$. 

Represent $\Omega$ as the union of $O(\e^{-1/2})$ intervals $U_n=\e^{1/2}[n,n+1]$. The first and last $U_n$ are truncated: $U_n\to U_n\cap\Omega$. Adding the estimates \eqref{var est (1)} with $\al=\e^{1/2}n$ gives
\be\label{var full int}
\TV(g_{m,\e}^{(1)},I)=\TV(g_{m,\e}^{(1)},\Omega)= \sum_{n:U_n\subset \Omega}\TV(g_{m,\e}^{(1)},U_n)\le c\rho(m)\e^{1/2}\ln(1/\e). 
\ee

\begin{remark} Estimate \eqref{var est (1)} is a key result in the proof of Lemma~\ref{lem:main-res-A} where we need the regularity of $H_0$ (assumption~\ref{ass:H0}(2)).
\end{remark}

\subsection{Estimate $\TV(g_{m,\e}^{(2)},I)$} A model integral gives
\be\label{model der int}\begin{split}
\biggl(\int_{-\infty}^{\al-\e^{1/2}}+\int_{\al+\e^{1/2}}^\infty\biggr)\frac{|\theta/\e|}{1+(\theta(\theta-\al)/\e)^2}\dd\theta\le c(1+\e^{-1/2}|\alpha|)^{-1},\ \al\in\br.
\end{split}
\ee
See appendix~\ref{ssec: two sep ints} for the proof. Here we use that 
\be\label{Xi1 Om}
\{\theta\in[-a,a]:|\CA(\theta)-\al|\le c\e^{1/2}\}\subset \Xi_1(\al,\e)
\ee
for some $c>0$ and then model $\CA(\theta)-\al$ with $\theta-\al$ when estimating integrals. 
Differentiating \eqref{three As} with $n=2$ with respect to $\al$ gives 
\be
\begin{split}\label{der g2 bnds}
(g_{m,\e}^{(2)}(\al))’&\le c\rho(m)\bigg[\frac{|m|\e^{1/2}}{1+\e^{-1/2}|\al|}+\frac{1}{1+\e^{-1/2}|\al|}+\frac{1}{(1+\e^{-1/2}|\al|)^2}\bigg]\\
&\le c\rho(m)(1+\e^{-1/2}|\al|)^{-1}.\end{split}
\ee
We used here that $|m|\le\e^{-\ga}$ implies $|m|\e^{1/2}=O(1)$. The first term in brackets is obtained by differentiating the exponential factor in front of the double integral (which brings out a factor $m$) and using \eqref{model int 0}. The second term in brackets is obtained by differentiating $R(\theta,\al)/\e$, where $R(\theta,\al)=s(\theta,\al)\theta$, in the argument of $\tilde\psi_m$. This brings out a factor $s/\e$ and then we use statement (2) of Lemma~\ref{aux lem sth} and \eqref{model der int}. Differentiating $h(\hat t,\theta,\al)$ in the argument of $\tilde\psi_m$ yields a lower order term.

Additionally, $\pa_\al s(\theta,\al)\asymp 1$ and the values of $\theta$ which solve $s(\theta,\al)=\pm\e^{1/2}$ satisfy $|\CA(\theta)-\al|\asymp \e^{1/2}$. Hence, such $\theta$ satisfy $\theta\in \CA^{-1}([\al-c\e^{1/2},\al+c\e^{1/2}])$. Combining this with Lemmas~\ref{lem:psi psider} and \ref{lem:aux props} gives the second term in brackets in \eqref{der g2 bnds}. 

Finally, \eqref{der g2 bnds} yields
\be
\label{g phi bnds 1}
\TV(g_{m,\e}^{(2)},I)\le c\rho(m)\int_{I}(1+\e^{-1/2}|\al|)^{-1}\dd\al\le c\rho(m)\e^{1/2}\ln(1/\e).
\ee

Note that only assumption~\ref{ass:H0}(1) is required for \eqref{der g2 bnds} to hold, assumption~\ref{ass:H0}(2) is not necessary.

\subsection{End of proof of Lemma~\ref{lem:apply lemma 1}}
With $\phi$ given in \eqref{Am small g}, we have $\phi_{\text{mx}}^{\prime\prime}=O(|m|)$. Using \eqref{eA 2nd} and adding \eqref{var full int} and \eqref{g phi bnds 1} yields
\be\label{main coef}\begin{split}
&\bigl(1+\e\phi_{\text{mx}}^{\prime\prime}\bigr)\TV(g_{m,\e},I)+\phi_{\text{mx}}^{\prime\prime}\int_{I}|g_{m,\e}(\al)|\dd\al\le c \rho(m)\e^{1/2}\ln(1/\e).
\end{split}
\ee
Here we have again used that $|m|\e^{1/2}=O(1)$. 

Finally, inequality \eqref{eA 1st} implies that $g_{m,\e}$ satisfies \eqref{g_e ass 1 v2}. 
From \eqref{var est (1)} and \eqref{der g2 bnds}, we see that $g_{m,\e}$ satisfies \eqref{g_e ass 2 v2}.

\section{Proof of Lemma~\ref{lem:exc cases}}\label{sec:exceptional}

Even though $\kappa\phi’(\al)$ may take the value $l_\star=\lfloor \kappa\phi’(\alst)\rfloor$ at two points (on either side of $\alst$, see Figure~\ref{fig:sine}), in this section we assume that $\al_{l_\star}^*$ is the smaller of the two (i.e., $\al_{l_\star}^*<\alst$).

The interval $I_{\star}\ni\alst$ is exceptional because the assumptions of Lemma~\ref{lem:int1 est} are not satisfied: $\phi_{\text{mn}}^{\prime\prime}(I_{\star})=0$. By \eqref{g_e ass 1 v2}, \eqref{g_e ass 2 v2},
\be\label{exc2 g_e ass}
|g_\e(\al)|\le c\rho(m)\e,\ \al\in I_\star\text{ and } 
\TV(g_\e,U)\le c\rho(m)\e 
\ee
for any interval $U\subset I_\star$ of length $|U|=O(\e^{1/2})$ because $|\al|\asymp 1$ if $\al\in I_\star$. 

Split $I_\star$ into two intervals : $I_\star^a:=[\al_{l_\star}^*,\al_\star]$ and $I_\star^b:=[\alst,2\alst-\al_{l_\star}^*]$. The two intervals are completely analogous, so we prove \eqref{exc II} by restricting the interior sum to $\al_k\in I_\star^a$. 

If $\alst-\al_{l_\star}^*\le \e^{1/2}$, direct estimation gives
\be
\label{exc II aux}
\sum_{\al_k\in I_\star^a} |g_{m,\e}(\al_k)|=O(\rho(m)\e^{1/2}),
\ee
and summation with respect to $m$ immediately proves \eqref{exc II} if $\bt>1$. Therefore, in what follows we assume $\alst-\al_{l_\star}^*>\e^{1/2}$.

If $\{\kappa \phi’(\alst)\}\le 1/2$, we use Lemma~\ref{lem:sum1 est} to reduce the sum with respect to $\al_k$ to an integral by using $l=l_\star$ in \eqref{h-fn}, \eqref{sum est}. If $\{\kappa \phi’(\alst)\}>1/2$, i.e. $\kappa \phi’(\alst)$ is closer to $l_\star+1$ than to $l_\star$, then we subdivide $I_\star^a$ into two intervals: $I_\star^a=[\al_{l_\star}^*,\al_{l_\star+1/2}^*]\cup[\al_{l_\star+1/2}^*,\al_\star]$, where $\al_{l_\star+1/2}^*$ satisfies $\kappa \phi’(\al_{l_\star+1/2}^*)=l_\star+1/2$. In this case, the sum with respect to $\al_k\in [\al_{l_\star}^*,\al_{l_\star+1/2}^*]$ is reduced to an integral over $[\al_{l_\star}^*,\al_{l_\star+1/2}^*]$ by using $l=l_\star$ in Lemma~\ref{lem:sum1 est} and the sum with respect to $\al_k\in [\al_{l_\star+1/2}^*,\al_\star]$ is reduced to an integral over $[\al_{l_\star+1/2}^*,\al_\star]$ by using $l=l_\star+1$. Note that Lemma~\ref{lem:sum1 est} applies regardless of whether there exists an $\al$ such that $\kappa\phi’(\al)=l$. In all cases, the magnitude of the error this replacement introduces is $O(\rho(m)\e^{1/2})$ for each $m$, and the sum over $m$ is $O(\e^{1/2})$ if \textcolor{black}{$\bt>1$.}

Set $\nu:=\al_{l_\star+1/2}^*$ if $\al_{l_\star+1/2}^*$ exists and $\nu:=\al_\star$ if $\al_{l_\star+1/2}^*$ does not exist.

\begin{lemma}\label{lem:exc cases alt1} Under the assumptions of Lemma~\ref{lem:exc cases}, one has
\be\begin{split}
\label{exc II alt1}
\sum_{1\le |m|\le\e^{-\ga}}&\biggl|\frac1{\Delta\al}\int_{\al_{l_\star}^*}^\nu (g_{m,\e} h_{l_\star})(\al) e\biggl(\frac{\phi_{l_\star}(\al)}\e\biggr)\dd \al\biggr|=O(\e^{1/2}\ln(1/\e))
\end{split}
\ee
if \textcolor{black}{$\bt>(\eta_\star+1)/2$.}
\end{lemma}

\begin{lemma}\label{lem:exc cases alt2} Under the assumptions of Lemma~\ref{lem:exc cases}, one has
\be
\label{exc II alt2}
\sum_{1\le |m|\le\e^{-\ga}}\biggl|\frac1{\Delta\al}\int_{\al_{l_\star+1/2}^*}^{\alst} (g_{m,\e} h_{l_\star})(\al) e\biggl(\frac{\phi_{l_\star}(\al)}\e\biggr)\dd \al\biggr|=O(\e^{1/2}\ln(1/\e))
\ee
if \textcolor{black}{$\bt>(\eta_\star/2)+1$.} The integral in \eqref{exc II alt2} is assumed to be zero for those values of $m$ for which $\al_{l_\star+1/2}^*$ does not exist.
\end{lemma}

The two lemmas prove Lemma~\ref{lem:exc cases}.

\subsection{Proof of Lemma~\ref{lem:exc cases alt1}}\label{ssec:exc1st}
Recall that $\alst$ satisfies $\phi^{\prime\prime}(\alst)=0$, $\alst$ is a local maximum of $\phi’(\al)$ and
$l_\star=\lfloor\kappa \phi’(\alst)\rfloor$. The dependence of $l_\star$ on $m$ is omitted for simplicity.  Consider the subintervals 
\be\label{exc2 U12}
U_1:=[\al_{l_\star}^*,\al_{l_\star}^*+\e^{1/2}],\ 
U_2:=[\al_{l_\star}^*+\e^{1/2},\nu],\ U_1\cup U_2=[\al_{l_\star}^*,\nu].
\ee
Let $J_n$, $n=1,2$, be defined as in \eqref{Jn def first} with $g_\e=g_{m,\e}$, $l=l_\star$ and $U_{1,2}$ defined in \eqref{exc2 U12}. Using \eqref{exc2 g_e ass}, direct estimation implies $|J_1|\le c\rho(m)\e^{1/2}$. Clearly,
\be\label{Jn-eps def}\begin{split}
|J_2|\le & \frac1{\Delta\al}\left|\int_{U_2} (g_{m,\e} h_{l_\star})(\al) e\biggl(\frac{\phi_{l_\star}(\al)}\e\biggr)\dd \al\right|\le c(J_{2,1}+J_{2,2}),\\
J_{2,1}:=&\biggl| \frac{g_{m,\e}(\al)}{\sin(\pi\kappa \phi_{l_\star}’(\al))}\bigr|_{\al_{l_\star}^*+\e^{1/2}}^{\nu}\biggr|,\ 
J_{2,2}:=\int_{U_2}\biggl|\pa_\al\frac{g_{m,\e}(\al)}{\sin(\pi\kappa \phi_{l_\star}’(\al))}\biggr|\dd\al,
\end{split}
\ee
where we have used the definition of $h_l$.

Clearly, $\phi(\al)=-m|x_0| \cos(\al-\al_{x_0})$ and $\phi’(\al)=m|x_0| \sin(\al-\al_{x_0})$, where $x_0=|x_0|(\cos\al_{x_0},\sin\al_{x_0})$. Since $\alst$ is a local maximum of $\phi’(\al)$, we can write 
\be\label{phi pr alt}
\kappa\phi’(\al)=\cx m\cos(\al-\alst),\ \cx:=\kappa|x_0|\text{sgn}(m). 
\ee
Using that $\kappa\phi’(\al_{l_\star}^*)=l_\star$ gives
\be\label{phi cos}\begin{split}
\kappa\phi_{l_\star}’(\al)=&\cx m\cos(\alst-\al)-l_\star\\
=&\cx m[\cos(\alst-\al)-\cos(\alst-\al_{l_\star}^*)],\ \al\in U_2.
\end{split}
\ee
Therefore,
\be\label{exc2 misc}\begin{split}
&\phi_{l_\star}’(\al)\asymp |m|\de(\al-\al_{l_\star}^*),\ \al\in U_2,\\
&\de:=\al_\star-\al_{l_\star}^*\ge c(\langle m\kappa\vec\al_\star^\perp\cdot x_0\rangle/|m|)^{1/2},
\end{split}
\ee
because $\de\le \de+(\alst-\al)\le 2\de$ when $\al\in U_2$.
By construction, $\kappa\phi’_{l_\star}(\nu)$ is farther away from the nearest integer than $\kappa\phi’_{l_\star}(\al_{l_\star}^*+\e^{1/2})$, i.e. 
\be\label{phi pr lb}
\langle\kappa\phi’_{l_\star}(\nu)\rangle\ge \langle\kappa\phi_{l_\star}’(\al_{l_\star}^*+\e^{1/2})\rangle.
\ee
Therefore, from \eqref{exc2 g_e ass}, \eqref{exc2 misc} and \eqref{phi pr lb},
\be\label{exc2 integr term}\begin{split}
J_{2,1}\le c \rho(m)\e\frac{1}{|m|\de\e^{1/2}}.
\end{split}
\ee

Consider now $J_{2,2}$. Clearly, $J_{2,2}\le J_{2,2,1}+J_{2,2,2}$, where
\be\label{B integr bnd 1}\begin{split}
J_{2,2,1}&:=\int_{U_2}\biggl|\frac{1}{\sin(\pi\kappa \phi_{l_\star}’(\al))}\pa_\al g_{m,\e}(\al)\biggr|\dd\al,\\
J_{2,2,2}&:=
\int_{U_2}\biggl|g_{m,\e}(\al)\pa_\al\frac{1}{\sin(\pi\kappa \phi_{l_\star}’(\al))}\biggr|\dd\al.
\end{split}
\ee
Set $N:=\lfloor(\nu-\al_{l_\star}^*)/\e^{1/2}\rfloor$ and split $U_2$ into $N$ subintervals 
\be\label{exc2 U2n}
U_{2,n}:=\al_{l_\star}^*+\e^{1/2}[n,n+1],\ n=1,2,\dots,N-1,\
U_{2,N}:=\al_{l_\star}^*+[\e^{1/2}N,\nu],
\ee
analogously to \eqref{sum est 1 v3}. Since $\nu-\al_{l_\star}^*$ is bounded, $N\le c\e^{-1/2}$. By \eqref{exc2 g_e ass}, the first line in \eqref{exc2 misc} and the definition of $U_{2,n}$,
\be\label{J21 bnd 1}\begin{split}
J_{2,2,1}\le & \sum_{n=1}^N \TV(g_{m,\e},U_{2,n})\max_{\al\in U_{2,n}}|\sin(\pi\kappa \phi_{l_\star}’(\al))|^{-1}\\
\le & c\rho(m)\e\sum_{n=1}^{N} \frac{1}{|m|\de\e^{1/2}n}\le c\frac{\rho(m)}{|m|\de}\e^{1/2}\ln(1/\e).
\end{split}
\ee
To estimate $J_{2,2,2}$ we again use \eqref{exc2 g_e ass}:
\be\label{J22 bnd 1}\begin{split}
J_{2,2,2}&\le
c\rho(m)\e\int_{U_2}\biggl|\pa_\al\frac{1}{\sin(\pi\kappa \phi_{l_\star}’(\al))}\biggr|\dd\al.
\end{split}
\ee
Since $0<\kappa\phi_{l_\star}’(\al)\le 1/2$ when $\al\in U_2$, the derivative in \eqref{J22 bnd 1} has a constant sign on $U_2$ and by \eqref{exc2 U12}, \eqref{exc2 misc} and \eqref{phi pr lb}
\be\label{J22 bnd 2}\begin{split}
J_{2,2,2}&\le
c \frac{\rho(m)\e}{|\sin(\pi\kappa \phi_{l_\star}’(\al_{l_\star}^*+\e^{1/2}))|}\le c\frac{\rho(m)}{|m|\de}\e^{1/2}.
\end{split}
\ee

Combining \eqref{exc2 integr term}, \eqref{J21 bnd 1} and \eqref{J22 bnd 2} with the estimate for $J_1$ and using \eqref{exc2 misc} gives
\be\label{exc2 Jn bnd}\begin{split}
|J_1|+|J_2|&\le c \frac{\rho(m)}{(|m|\langle m\kappa\vec\al_\star^\perp\cdot x_0\rangle)^{1/2}}\e^{1/2}\ln(1/\e).
\end{split}
\ee
Adding the estimates \eqref{exc2 Jn bnd} for all $1\le|m|\le \e^{-\ga}$ gives $O(\e^{1/2}\ln(1/\e))$ provided that \textcolor{black}{$\bt>(\eta_\star+1)/2$.  }

\subsection{Proof of Lemma~\ref{lem:exc cases alt2}} 
We use \eqref{Jn def first} with $g_\e=g_{m,\e}$ and $l=l_\star+1$. Clearly, there is no $J_1$. We need to estimate only $J_2$, where $U_2:=[\al_{l_\star+1/2}^*,\alst]$:
\be\label{Jn-eps def v2}\begin{split}
|J_2|&\le c\biggl| \frac{g_{m,\e}(\al)}{\sin(\pi\kappa \phi_{l_\star+1}’(\al))}\bigr|_{\al_{l_\star+1/2}^*}^{\alst}\biggr|
+c\int_{\al_{l_\star+1/2}^*}^{\alst}\biggl|\pa_\al\frac{g_{m,\e}(\al)}{\sin(\pi\kappa \phi_{l_\star+1}’(\al))}\biggr|\dd\al.
\end{split}
\ee
Since $|\kappa \phi_{l_\star+1}’(\al_{l_\star+1/2}^*)|=1/2$, the integrated term is bounded by
\be\label{exc2 integr term v2}\begin{split}
c \rho(m)\e\langle m\kappa\vec\al_\star^\perp\cdot x_0\rangle^{-1}.
\end{split}
\ee

To estimate the integral, we argue similarly to \eqref{B integr bnd 1}--\eqref{exc2 Jn bnd}. The analogue of \eqref{exc2 U2n} becomes
\be\label{exc2 U2n v2}\begin{split}
U_{2,n}:=&\alst-\e^{1/2}[n-1,n],\ n=1,2,\dots,N-1,\\
U_{2,N}:=&[\alst-\e^{1/2}(N-1),\al_{l_\star+1/2}^*],\\ 
N=&\lceil(\alst-\al_{l_\star+1/2}^*)\e^{-1/2}\rceil=O(\e^{-1/2}).
\end{split}
\ee
Similarly to \eqref{phi cos}, \eqref{exc2 misc}, 
\be\label{phi cos v2}\begin{split}
-\kappa\phi_{l_\star+1}’(\al)=&\lceil \cx m\rceil-\cx m\cos(\alst-\al)\\
=&(1-\{Qm\})+\cx m(1-\cos(\alst-\al))],\ \al\in [\al_{l_\star+1/2}^*,\alst].
\end{split}
\ee
Recall that there is no $\al$ such that $\kappa\phi’(\al)=l_\star+1$. Therefore,
\be\label{exc2 misc v2}\begin{split}
&-\kappa\phi_{l_\star+1}’(\al)\ge \langle Qm\rangle+c|m|(\alst-\al)^2,\ \al\in [\al_{l_\star+1/2}^*,\alst].
\end{split}
\ee
Similarly to \eqref{J21 bnd 1},
\be\label{J21 bnd 1 v2}\begin{split}
J_{2,1}\le & \sum_{n=1}^N \TV(g_{m,\e},U_{2,n})\max_{\al\in U_{2,n}}|\sin(\pi\kappa \phi_{l_\star+1}’(\al))|^{-1}\\
\le & c\rho(m)\e\sum_{n=0}^{N-1} \frac{1}{\langle m\kappa\vec\al_\star^\perp\cdot x_0\rangle+|m|\e n^2}.
\end{split}
\ee
To estimate the last sum, denote $r^2:=\langle m\kappa\vec\al_\star^\perp\cdot x_0\rangle/|m|$. Then
\be
\e^{1/2}\sum_{n=0}^{N-1} \frac{1}{r^2+\e n^2}\le c\int_0^\infty \frac{\dd x}{x^2+r^2}\le c/r.
\ee
Hence
\be\label{J21 bnd 1 v3}\begin{split}
J_{2,1}\le  c\rho(m)\e^{1/2} (|m|\langle m\kappa\vec\al_\star^\perp\cdot x_0\rangle)^{-1/2}.
\end{split}
\ee
The analogue of \eqref{J22 bnd 2} becomes
\be\label{J22 bnd 2 v2}\begin{split}
J_{2,2}&\le
c \frac{\rho(m)\e}{|\sin(\pi\kappa \phi_{l_\star+1}’(\alst))|}\le c\frac{\rho(m)\e}{\langle m\kappa\vec\al_\star^\perp\cdot x_0\rangle},
\end{split}
\ee
where we have used \eqref{phi cos v2} with $\al=\alst$.

The estimates \eqref{exc2 integr term v2} in \eqref{J22 bnd 2 v2} are the same. Summing them over $1\le|m|\le\e^{-\ga}$ and using that they contain an “extra” factor $\e^{1/2}$ gives the sum $O(\e^{(1/2)-\ga(\eta_\star-\bt+1)})$. Recall that $\ga=1/(2(\bt-1))$. The exponent is positive if \textcolor{black}{$\bt>(\eta_\star/2)+1$.}
The estimate \eqref{J21 bnd 1 v3} is similar to \eqref{exc2 Jn bnd}, so it gives \textcolor{black}{$\bt>(\eta_\star+1)/2$.}

\section{Proofs of Lemmas~\ref{lem:generic} and \ref{lem:ineq II}}\label{sec:generic}

\subsection{Proof of Lemma~\ref{lem:generic}}\label{ssec:generic}

For each $m\not=0$, there are at most $O(|m|)$ integers $l\in\mathcal L(m)$ (cf. the paragraph following \eqref{exc intervals}). By \eqref{int1 est}, the contribution of any finite number of terms to the sum in \eqref{sum-all} is still of order $O(\e^{1/2}\ln(1/\e))$. Hence, in what follows we assume $|m|\gg 1$. Pick some $0<\de_1\ll 1$, and consider three sets
\be\label{L sets}\begin{split}
&L_1(m):=\{l\in \mathcal L(m): l=\kappa\phi’(\al),\ |\al|\le\de_1\},\\
&L_2(m):=\{l\in \mathcal L(m): \kappa |x_0|(1-\de_1)\le |l/m| \le \kappa |x_0|\},\\
&L_3(m):=\{l\in \mathcal L(m): |l/m|< \kappa |x_0|(1-\de_1),\ l\not\in L_1(m)\}.
\end{split}
\ee
There are $O(|m|)$ elements in each $L_n(m)$. 

In what follows we apply \eqref{int1 est} with $g_\e=g_{m,\e}$. Recall that the right side of \eqref{int1 est} is denoted $W_{l,m}(\e)$. We sum $W_{l,m}(\e)$ over $l\in L_n(m)$, $n=1,2,3$, and then over $m$. 
All the sums with respect to $m$ are assumed to be over $M\le |m|\le \e^{-\ga}$ for some $M\gg 1$. 

\noindent
\underline{If $l\in L_3(m)$}, one has $\phi_{\text{mn}}^{\prime\prime}(I_l)\asymp |m|$, $|\al_l^*|\asymp 1$, and Lemma~\ref{lem:int1 est} implies
\be\label{sum 3}
\sum_{m}\sum_{l\in L_3(m)} W_{l,m}(\e)=\e^{1/2}\ln(1/\e)\sum_{m}O(\rho(m))=O(\e^{1/2}\ln(1/\e)),\
\textcolor{black}{\bt>1.}
\ee

\noindent
\underline{If $l\in L_1(m)$}, one has $\phi_{\text{mn}}^{\prime\prime}(I_l)\asymp |m|$, so we look at the sum $\sum_{l\in L_1(m)}|\al_l^*|^{-1}$. 
Set $l_0=\lfloor \kappa\phi’(0)\rfloor$. Then, 
\be\label{al lower bnd}
c|m\al|\ge |\kappa\phi’(\al)-\kappa\phi’(0)|=|\kappa\phi’(\al)-l_0-\{\kappa\phi’(0)\}|,\ 
|\al|\le\de_1. 
\ee
This implies
\be\label{ang0 bnd}
|\al_{l_0}^*|,|\al_{l_0+1}^*|\ge c \langle m\kappa\vec\al_0^\perp\cdot x_0\rangle/|m|
\ee
and
\be\label{sum 11}\begin{split}
\sum_{l\in L_1(m)}\frac1{|\al_l^*|}&\le c\sum_{l\in L_1(m)}\frac{|m|}{\vert l-l_0-\{\kappa\phi’(0)\}\vert}\\
&\le c|m|\bigl(\langle m\kappa\vec\al_0^\perp\cdot x_0\rangle^{-1}+O(\ln|m|)\bigr).
\end{split}
\ee
The first term in parentheses dominates the second, so Lemma~\ref{lem:int1 est} yields
\be\label{sum 12}\begin{split}
\sum_{m}\sum_{l\in L_1(m)} W_{l,m}(\e)&=O(\e^{1/2}\ln(1/\e))\sum_{m}\frac{\rho(m)|m|}{\langle m\kappa\vec\al_0^\perp\cdot x_0\rangle}\\
&=O(\e^{1/2}\ln(1/\e)),\ \textcolor{black}{\bt>\eta_0+2.}
\end{split}
\ee

\noindent
\underline{If $l\in L_2(m)$}, one has $\al_l^*\asymp 1$, so we look at the sums $\sum_{l\in L_2(m)}(\phi_{\text{mn}}^{\prime\prime}(I_l))^{-n}$, $n=1,2$. 
By construction (see Figure~\ref{fig:sine}),
\be
I_l=\begin{cases}
[\al_{l-(1/2)}^*,\al_{l+(1/2)}^*],& l\in\mathcal L(m),l\not= l_\star,\\
[\al_{l-(1/2)}^*,\al_l^*],& l= l_\star.
\end{cases} 
\ee
Here $\al_r^*$ is a locally unique solution of $\kappa\phi^{\prime}(\al)=r$, e.g. $\al_{l+(1/2)}^*$ is between $\al_l^*$ and $\al_{l+1}^*$. By \eqref{phi pr alt}, $|\kappa\phi^{\prime\prime}(\al)|=[(\cx m)^2-(\kappa\phi’(\al))^2]^{1/2}$. Therefore,
\be\label{phmn ineq}
\phi_{\text{mn}}^{\prime\prime}(I_l)\ge c|m|^{1/2}
\begin{cases}
(l_{\star}-(|l|+1/2))^{1/2},& l\in\mathcal L(m),l\not=\l_\star,\\
\langle m\kappa\vec\al_\star^\perp\cdot x_0\rangle^{1/2},& l=l_\star.
\end{cases}
\ee
Then,
\be\label{sum 21}\begin{split}
&\sum_{l\in L_2(m)}\frac1{(\phi_{\text{mn}}^{\prime\prime}(I_l))^n}\\
&\le \frac{c}{|m|^{n/2}}\biggl[\langle m\kappa\vec\al_\star^\perp\cdot x_0\rangle^{-n/2}+\sum_{l=0}^{ l_\star-1}(l_{\star}-(l+1/2))^{-n/2}\biggr]\\
&\le c\bigl[|m|\langle m\kappa\vec\al_\star^\perp\cdot x_0\rangle\bigr]^{-n/2},\ n=1,2,
\end{split}
\ee
and 
\be\label{sum 22}\begin{split}
\sum_{m}&\rho(m)\biggl[|m|+\frac{\ln(1/\e)}{(|m|\langle m\kappa\vec\al_\star^\perp\cdot x_0\rangle)^{1/2}}
+\frac{1}{\langle m\kappa\vec\al_\star^\perp\cdot x_0\rangle}\biggr]\\
&=O(\ln(1/\e)) \text{  if  } \textcolor{black}{\bt>\eta_\star+1.}
\end{split}
\ee
Consequently, 
\be\label{sum 23}
\sum_{m}\sum_{l\in L_2(m)} W_{l,m}(\e)=O(\e^{1/2}\ln(1/\e))\text{  if  } \textcolor{black}{\bt>\eta_\star+1.}
\ee
Combining \eqref{sum 3}, \eqref{sum 12} and \eqref{sum 23} finishes the proof.

\subsection{Proof of Lemma~\ref{lem:ineq II}}\label{ssec:ineq II}
Clearly,
\be\label{prf1 et II}\begin{split}
\sum_{|\al_k|\le\pi/2} & \int_{|\al-\al_k|\le \Delta\al/2}|A_0(\al,\e)-A_0(\al_k,\e)|\dd\al\\
&\le \Delta\al\sum_{|\al_k|\le\pi/2}\max_{\al:|\al-\al_k|\le \Delta\al/2}|A_0(\al,\e)-A_0(\al_k,\e)|\\
&\le c \TV(g_{0,\e},[-\pi/2,\pi/2])=O(\e^{1/2}\ln(1/\e)).
\end{split}
\ee
where we have used \eqref{Am small g} and \eqref{main coef orig} with $m=0$ and $I=[-\pi/2,\pi/2]$. 
In this case, $\phi\equiv0$ and $\phi_{\text{mx}}^{\prime\prime}=\phi_{\text{mn}}^{\prime\prime}=0$.

\section{Proofs of Lemmas in section~\ref{sec:caseB}}\label{sec:caseB proofs}

\subsection{Proof of Lemma~\ref{lem:aux props rem}}

Properties \eqref{A props rem} follow immediately from an easy computation in coordinates. To prove \eqref{r props rem} note that
\be
b(\theta,\al)=\frac{\al\cdot(y(\theta)-x_0)}{\CA(\theta)-\al}=\pm|y(\theta)-x_0|\frac{\sin(\CA(\theta)-\al)}{\CA(\theta)-\al}.
\ee
The sign (plus or minus) is selected based on the location of $x_0$ relative to $y(0)$. For example, the sign is a plus in the arrangement depicted in Figure~\ref{fig:rem A_fn}. The desired assertion is now obvious.

\subsection{Proof of Lemma~\ref{lem:rem A-to-g2}}

\begin{lemma}\label{lem:three cases} One has
\be\label{rem mod Omega}\begin{split}
&\int_{-a}^a\frac{\dd\theta}{1+((\theta^2-\al)/\e)^2}\le c
\begin{cases}
\e\bigl(1+(\al-a^2)/\e\bigr)^{-1},& \al\ge a^2,\\
\e^{1/2},& 0\le\al\le 4\e,\\
\e^{1/2}\bigl(1+|\al|/\e\bigr)^{-3/2},& \al\le 0.
\end{cases}
\end{split}
\ee
\end{lemma}
See subsection~\ref{ssec:three cases} for the proof. 
By \eqref{A simpl orig} and Lemmas~\ref{lem:aux props rem} and \ref{lem:three cases}, an easy calculation shows
\be\label{A-outside}\begin{split}
\Delta\al &\sum_{m\in\BZ}\biggl(\sum_{\al_k\le 0}+\sum_{0\le\al_k\le 4\e}+\sum_{\al_{\text{mx}}\le\al_k\le\pi/2 }\biggr)|A_m(\al_k,\e)|\\
&\le c \e^{1/2}\sum_{m\in\BZ}\rho(m)=O(\e^{1/2}),\ \textcolor{black}{\bt>1.}
\end{split}
\ee
Hence the range of $\al_k$ in the sum on the left in \eqref{extra terms I} can be restricted to the interval $\Omega_\e=[4\e,\al_{\text{mx}}]$.

Next, we reduce the range of $m$ and simplify $A_m(\al,\e)$. Introduce two sets 
\be\label{rem two mod sets}\begin{split}
&\Xi_1’(\al,\e):=\{\theta\in\br: |\theta^2-\al| \ge (\e\al)^{1/2}\},\\
&\Xi_2’(\al,\e):=\{\theta\in\br: |\theta^2-\al| \le (\e\al)^{1/2}\},\ 
4\e\le \al\le\pi/2,
\end{split}
\ee
and the corresponding integrals
\be\label{rem mod1}\begin{split}
&\Phi_n(\al):=\int_{\Xi_n’(\al,\e)}\frac{\dd\theta}{1+((\theta^2-\al)/\e)^2},\ 4\e\le \al\le\pi/2,\ n=1,2.
\end{split}
\ee
If either $\Xi_1’(\al,\e)=\varnothing$ or $\Xi_2’(\al,\e)=\varnothing$, then the corresponding $\Phi_n(\al)$ is assumed to be zero. The sets $\Xi_1’(\al,\e)$ and $\Xi_2’(\al,\e)$ model the sets 
\be\label{rem two sets}\begin{split}
&\Xi_1(\al,\e):=\{\theta\in[-a,a]: |\CA(\theta)-\al| \ge (\e\al)^{1/2}\},\\
&\Xi_2(\al,\e):=\{\theta\in[-a,a]: |\CA(\theta)-\al| \le (\e\al)^{1/2}\},\ 
4\e\le \al\le\pi/2,
\end{split}
\ee
respectively.
The domain of integration with respect to $\theta$ in \eqref{Am small g rem} is $\Xi_2(\al,\e)$.

Straightforward calculations give (see subsection~\ref{ssec:Phi1_2} for the proof):
\be\label{rem mod2}\begin{split}
&\Phi_1(\al)=
O(\e^{3/2}/\al),\
\Phi_2(\al)=O(\e/\al^{1/2}),\ \al\ge 4\e.
\end{split}
\ee

Since $\Phi_2$ dominates $\Phi_1$, we compute
\be\begin{split}\label{extra extra rem}
\Delta\al\sum_{|m|\ge \e^{-\ga}} \sum_{\al_k\in\Omega_\e}|A_m(\al_k,\e)|
&\le c\sum_{|m|\ge \e^{-\ga}}\rho(m)\sum_{1\le k\le O(1/\e)} \frac{\e}{\al_k^{1/2}}\\
&=O(\e^{1/2}),\ \ga=1/(2(\bt-1)),
\end{split}
\ee
\textcolor{black}{if $\bt>1$}. This completes the proof of the top statement in \eqref{rem sums A-g22}. In what follows we assume $|m|\le \e^{-\ga}$.

Next we split $A_m$ into two parts $A_m^{(n)}$, which are obtained by integrating with respect to $\theta$ over $\Xi_n(\al,\e)$, $n=1,2$, respectively, in \eqref{A simpl orig}. If either $\Xi_1(\al,\e)=\varnothing$ or $\Xi_2(\al,\e)=\varnothing$, the corresponding $A_m^{(n)}$ is assumed to be zero. By combining Lemma~\ref{lem:aux props rem} and \eqref{rem mod1}, \eqref{rem mod2}, we find from \eqref{extra terms I}, \eqref{A simpl orig} 
\be\begin{split}\label{small terms rem}
\Delta\al\sum_{|m|\le \e^{-\ga}} \sum_{\al_k\in\Omega_\e} |A_m^{(1)}(\al_k,\e)|
&\le c\sum_{|m|\le \e^{-\ga}}\rho(m)\sum_{1\le k\le \e^{-1}} \frac{\e^{3/2}}{\al_k}\\
&=O(\e^{1/2}\ln(1/\e)),\ \textcolor{black}{\bt>1,}
\end{split}
\ee
This proves the bottom statement in \eqref{rem sums A-g22}. Lemma~\ref{lem:rem A-to-g2} is proven by noting that $g_{m,\e}^{(2)}=\e A_m^{(2)}$ (see \eqref{Am small g rem}).

\subsection{Proof of Lemma~\ref{lem:case B last}}\label{ssec:prf case B}

With some abuse of notation, in what follows we compute $g_{m,\e}^{(2)}$ by restricting the domain of integration with respect to $\theta$ in \eqref{Am small g rem} to the subset $\{\theta\in[0,a]:|\CA(\theta)-\al| \le (\e\al)^{1/2}\}$. The contribution involving negative $\theta$ can be estimated the same way. 

First, we simplify $g_{m,\e}^{(2)}$. Change variable $\theta\to s=\CA(\theta)-\al$ and represent $g_{m,\e}^{(2)}$ in the form
\be\label{rem gme 1}\begin{split}
g_{m,\e}^{(2)}(\al)=&\int_{|s|\le(\e\al)^{1/2}}\int_0^{\mu(\al,s)} G(\al,s,\hat t,\e)\dd\hat t \dd s,\\
G(\al,s,\hat t,\e)=&e\left(-m \vec\al\cdot \check x\right)\tilde\psi_m\bigl(b(\theta,\al)(s/\e) +h(\hat t,\theta,\al)\bigr)
F(\theta,\e \hat t)\,\pa_s\theta(s,\al),\\
\mu(\al,s):=&H_0(\e^{-1/2}\theta),\ \theta=\theta(s,\al):=\CA^{-1}(s+\al),\ 
\al\in\Omega_\e,
\end{split}
\ee
where $b(\theta,\al)$ is the same as in \eqref{r props rem}. The dependence of some functions on $\e$ is suppressed for simplicity. The unique inverse $\CA^{-1}$ is obtained by restricting its range to $\theta\ge 0$ (see the first paragraph in this subsection). By \eqref{A props rem}, $(\CA^{-1})’(t)\asymp t^{-1/2}$, $t>0$. Combined with Lemma~\ref{lem:aux props rem} this implies
\be\label{G 1 est}
|G(\al,s,\hat t,\e)|\le c\rho(m)\big[(1+(s/\e)^2)(s+\al)^{1/2}\big]^{-1},\ 
\al\in\Omega_\e.
\ee

Set $N:=\lfloor(\al_{\text{mx}}/\e)^{1/2}\rfloor$, define the intervals
\be\label{rem I aux}
U_n:=\e[n^2,(n+1)^2],\ 2\le n\le N-1,\ U_N:=[\e n^2,\al_{\text{mx}}],
\ee
and the integrals
\be\label{rem gme 3}\begin{split}
g_{m,\e}^{(2,2)}(\al):=&\int_{|s|\le \e n}\int_0^{\mu(\al,s)} G(\e n^2,s,\hat t,\e)\dd\hat t \dd s,\ 
\al\in U_n,\ 2\le n\le N.
\end{split}
\ee
By construction, $\cup_{n=2}^N U_n=\Omega_\e$.
The functions $g_{m,\e}^{(2,2)}$ are obtained from $g_{m,\e}^{(2)}$ by replacing $\al$ with $\e n^2$ in the limits of integration with respect to $s$ and in the first argument of $G$ in \eqref{rem gme 1} (i.e., everywhere except in the arguments of $\mu$).

\begin{lemma}\label{lem:rem g-to-g22} Under the assumptions of Lemma~\ref{lem:case B last}, one has if \textcolor{black}{$\bt>1$}
\be\label{rem generic sum g22}\begin{split}
\sum_{|m|\le \e^{-\ga}}\sum_{n=2}^N\sum_{\al_k\in U_n}|g_{m,\e}^{(2)}(\al_k)-g_{m,\e}^{(2,2)}(\al_k)|=O(\e^{1/2}\ln(1/\e)).
\end{split}
\ee
\end{lemma}

By \eqref{rem generic sum g22} we can replace $g_{m,\e}^{(2)}$ with $g_{m,\e}^{(2,2)}$ in \eqref{rem sums A-g22} (the sum with respect to $k$ is understood as written in \eqref{rem generic sum g22}) and it suffices to apply Lemma~\ref{lem:sum1 est} to $g_{m,\e}^{(2,2)}$ on each $U_n$.

\begin{lemma}\label{lem:rem gme} Under the assumptions of Lemma~\ref{lem:case B last}, one has:
\be\label{g_int rem}\begin{split}
&|g_{m,\e}^{(2,2)}(\al)|\le c\rho(m)\e/\al^{1/2},\ \al\in U_n;\
\int_{U_n}|g_{m,\e}^{(2,2)}(\al)|\dd\al \le c\rho(m)\e^{3/2};\\
&\TV(g_{m,\e}^{(2,2)},U_n)\le c\rho(m)\e^{1/2}/n;\quad 2\le n\le N,\ m\in\mathbb Z.
\end{split}
\ee
\end{lemma}

From the first inequality in \eqref{g_int rem} it follows that 
\be\label{rem In sum g22}
\sum_{\al_k\in U_n} |g_{m,\e}^{(2,2)}(\al_k)|\le c\rho(m) \sum_{\al_k\in \e[n^2,(n+1)^2]}\e/\al_k^{1/2}\le c\rho(m)\e^{1/2}.
\ee
After summing \eqref{rem In sum g22} over $|m|\le\e^{-\ga}$, we conclude that in the sum in \eqref{rem sum last} we can ignore all $\al_k\in U_n$ for any finitely many intervals $U_n$. 

The following lemma finishes the proof of \eqref{extra terms I} in case (B). Recall that the intervals $I_l$ are defined in \eqref{main intervals}.

\begin{lemma}\label{lem:rem sum gme} Under the assumptions of Lemma~\ref{lem:case B last}, one has:
\be\label{rem sum gme22}\begin{split}
\sum_{1\le |m|\le \e^{-\ga}}\sum_{n=2}^N\biggl|\sum_{\al_k\in U_n}g_{m,\e}^{(2,2)}(\al_k) e(\phi(\al_k)/\e)\biggr|=O(\e^{1/2}\ln(1/\e))
\end{split}
\ee
if \textcolor{black}{$\bt>\eta_0+1$}.
\end{lemma}

\subsection{Proof of Lemma~\ref{lem:rem g-to-g22}}

Replace $\al$ with $\e n^2$ in the limits of integration with respect to $s$ in \eqref{rem gme 1}:
\be\label{rem gme 2}\begin{split}
g_{m,\e}^{(2,1)}(\al):=&\int_{|s|\le \e n}\int_0^{\mu(\al,s)} G(\al,s,\hat t,\e)\dd\hat t \dd s,\ 
\al\in U_n,\ 2\le n\le N.
\end{split}
\ee
This substitution makes $g_{m,\e}^{(2,1)}$ dependent on $n$. Moreover,
\be\label{del gme 2}\begin{split}
|g_{m,\e}^{(2)}(\al)-g_{m,\e}^{(2,1)}(\al)|\le &\int_{\e n\le |s|\le (\e \al)^{1/2}}\int_0^{\mu(\al,s)} |G(\al,s,\hat t,\e)|\dd\hat t \dd s\\ 
\le & c\rho(m)\int_{\e n\le |s|\le (\e \al)^{1/2}} \frac{\dd s}{(1+(s/\e)^2)(\al+s)^{1/2}}\\
\le & c\rho(m)\e^{1/2}/{n^3},\ \al\in U_n.
\end{split}
\ee

There are $O(\e^{-1/2})$ intervals $U_n\subset \Omega_\e$ and $O(n)$ values of $k$ such that $\al_k\in U_n$. Hence, replacing $g_{m,\e}^{(2)}(\al_k)$, $\al_k\in U_n$, with $g_{m,\e}^{(2,1)}(\al_k)$, $\al_k\in U_n$, in \eqref{rem generic sum g22} leads to
\be\label{lim repl 1}\begin{split}
&\sum_{|m|\le \e^{-\ga}}\sum_{n=2}^N\sum_{\al_k\in U_n}|g_{m,\e}^{(2)}(\al_k)-g_{m,\e}^{(2,1)}(\al_k)|\\
&\le c
\sum_{|m|\le \e^{-\ga}}\rho(m)\sum_{n=2}^{\e^{-1/2}} n\frac{\e^{1/2}}{n^3}=O(\e^{1/2})\sum_{|m|\le \e^{-\ga}}\rho(m)=O(\e^{1/2}),\ \textcolor{black}{\bt>1.}
\end{split}
\ee

Further simplification of \eqref{rem gme 2} is achieved by replacing $\al$ with $\e n^2$ in the argument of $G$ to obtain \eqref{rem gme 3}.
To estimate the associated error, we estimate the derivative $\pa_\al G(\al,s,\hat t,\e)$. Using that $|b|\asymp 1$, the derivatives $\pa_\theta b$ and $\pa_\al b$ are bounded and $(\CA^{-1})^{(j)}(t)\asymp |t|^{(1/2)-j}$, $j=1,2$, (see Lemma~\ref{lem:aux props rem}), we find from \eqref{rem gme 1}:
\be\label{rem Gder bnd}\begin{split}
&\max_{\al\in U_n,|\hat t|\le c}|\pa_\al G(\al,s,\hat t,\e)|
\le c\frac{\rho(m)}{(1+(s/\e)^2)(\al+s)^{1/2}}\\
&\times\biggl[|m|+\biggl(1+\frac{|s|}\e\biggr)\biggl(1+\frac{1}{(\al+s)^{1/2}}\biggr)+\frac{1}{(\al+s)^{1/2}}+\frac{1}{(\al+s)}\biggr],\\
&|s|\le\e n.
\end{split}
\ee
Since $\al+s\asymp \e n^2$, $|s|\le\e n$ and $\al\in U_n$, and $2\le n\le N=O(\e^{-1/2})$, \eqref{rem Gder bnd} simplifies to
\be\label{rem Gder bnd final}\begin{split}
&\max_{\al\in U_n,|\hat t|\le c}|\pa_\al G(\al,s,\hat t,\e)|\\
&\le c\frac{\rho(m)}{(1+(s/\e)^2)\e^{1/2} n}
\biggl[|m|+\frac{|s|}\e\frac1{\e^{1/2}n}+\frac{1}{\e n^2}\biggr],\ 
|s|\le\e n,\ 2\le n\le N.
\end{split}
\ee

Similarly to \eqref{del gme 2}, using \eqref{rem Gder bnd final} and that $|U_n|\asymp\e n$ gives 
\be\label{del gme 3}\begin{split}
&|g_{m,\e}^{(2,1)}(\al)-g_{m,\e}^{(2,2)}(\al)|\\
&\le\int_{|s|\le \e n}\int_0^{\mu(\al,s)} |G(\al,s,\hat t,\e)-G(\e n^2,s,\hat t,\e)|\dd\hat t \dd s\\ 
&\le c \e n\int_{|s|\le \e n} \max_{\al’\in U_n,|\hat t|\le c}|\pa_\al G(\al’,s,\hat t,\e)| \dd s\\ 
&\le  c\rho(m)\bigl(\e^{3/2}|m|+\e (\ln n/n)+\e^{1/2}n^{-2}\bigr),\ \al\in U_n,\ 2\le n\le N.
\end{split}
\ee
Similarly to \eqref{lim repl 1},
\be\label{lim repl 2}\begin{split}
&\sum_{|m|\le \e^{-\ga}}\sum_{n=2}^N\sum_{\al_k\in U_n}|g_{m,\e}^{(2,1)}(\al_k)-g_{m,\e}^{(2,2)}(\al_k)|\\
&\le c
\sum_{|m|\le \e^{-\ga}}\rho(m)\sum_{n=2}^{\e^{-1/2}} \bigl(\e^{3/2}|m|+\e(\ln n/n)+\e^{1/2}n^{-2}\bigr)=O(\e^{1/2}),\ \textcolor{black}{\bt>1,}
\end{split}
\ee
because $\e^{1/2}|m|\le 1$.

\subsection{Proof of Lemma~\ref{lem:rem gme}}\label{ssec:rem g22 props}

Again, using that $\e n^2+s\asymp \e n^2$ if $|s|\le\e n$, $n\ge 2$, \eqref{rem gme 1} and \eqref{G 1 est} imply $|g_{m,\e}^{(2,2)}(\al)|\le c\rho(m)\e^{1/2}/n$, $\al\in U_n$. Combining with $|U_n|\asymp\e n$ gives
\be\label{g_int rem v2}
\int_{U_n}|g_{m,\e}^{(2,2)}(\al)|\dd\al = O(\e^{3/2}).
\ee
Recall that the sets $U_n$ are defined in \eqref{rem I aux}. 

It remains to estimate $\TV(g_{m,\e}^{(2,2)},U_n)$. 
We argue similarly to section~\ref{ssec:step2}. 
By comparing \eqref{rem gme 3} and \eqref{small g v1 alt}, it is clear that \eqref{g findif}--\eqref{partition} still apply to $g_{m,\e}^{(2,2)}$ (with $\al=\e n^2$ in the arguments of $G$ and the domain of $s$-integration $|s|\le\e n$, cf. \eqref{rem gme 3}). By \eqref{G 1 est}, the analogue of \eqref{H0 rel est} becomes
\be\label{rem H0 rel est}\begin{split}
\sum_{k=1}^K |J_1(\al_{k-1}’,\al_k’)|&\le
c \rho(m)\int_{|s|\le \e n}\sum_{k=1}^K \frac{|H_0(\tilde\theta_k’)-H_0(\tilde\theta_{k-1}’)|}{(1+(s/\e)^2)(s+\e n^2)^{1/2}}\dd s\\
&\le
c \frac{\rho(m)}{\e^{1/2}n}\int_{|s|\le \e n}\TV(H_0,\tilde U_n(s))\frac{\dd s}{1+(s/\e)^2},\\
\tilde U_n(s):&=\e^{-1/2}\CA^{-1}(s+U_n),\ \tilde\theta_k’:=\e^{-1/2}\CA^{-1}(s+\al_k’),\\
\e n^2 &\le \al_0’<\al_1’<\dots<\al_K’\le \e(n+1)^2.
\end{split}
\ee
Here, the quantity $J_1$ is defined analogously to \eqref{g findif}. There is no $J_2$ now, since $G$ in \eqref{rem gme 3} does not depend on $\al$. As is easily seen,
\be\label{rem aux stuff}\begin{split}
&|\tilde U_n(s)|\asymp (n^2+n)^{1/2}-(n^2-n)^{1/2}\asymp 1,\ |s|\le \e n.
\end{split}
\ee
By \eqref{rem aux stuff} and assumption~\ref{ass:H0}(2), we can replace $\TV(H_0,\tilde U_n(s))$ with $V_0$ in \eqref{rem H0 rel est}. Therefore, 
\be\label{rem J1 sum}
\TV(g_{m,\e}^{(2,2)},U_n)\le c\sup_{K,\{\al_k’\}}\sum_{k=1}^K |J_1(\al_{k-1}’,\al_k’)|\le c\rho(m)\e^{1/2}/n.
\ee

\begin{remark} Estimate \eqref{rem J1 sum} is a key result in the proof of Lemma~\ref{lem:main-res-BC} (case (B)) where we need assumption~\ref{ass:H0}(2).
\end{remark}

\subsection{Proof of Lemma~\ref{lem:rem sum gme}}\label{ssec:lem prf}
Recall that the proof of Lemma~\ref{lem:main-res-BC} in case (B) follows the five steps described after \eqref{extra terms II}. The results obtained in this section up to this point complete the first three steps. In this subsection we complete steps (4) and (5).

\underline{Step (4).} By \eqref{sum est} (with $g=g_{m,\e}^{(2,2)}$), \eqref{g_int rem v2} and \eqref{rem J1 sum}, the total error of replacing the sum over $\al_k\in U_n$ in \eqref{rem sum gme22} with an integral over $U_n$ for all $U_n$ is bounded by
\be
c\sum_{|m|\le\e^{-\ga}}\rho(m)\sum_{n=2}^{\e^{-1/2}}\bigl[(1+\e |m|)\frac{\e^{1/2}}{n}+|m|\e^{3/2}\bigr]
=O(\e^{1/2}\ln(1/\e)),\ \textcolor{black}{\bt>1.}
\ee
This justifies replacing the sums over $U_n$ with integrals over $U_n$. 

\underline{Step (5).} Next we estimate the integrals over $U_n$. In this step $m\not=0$. To simplify the estimation, we group the $U_n$’s based on which interval $I_l$ they are a subset of: 
\be\label{Un 2 sums}
\sum_{n=2}^N\sum_{\al_k\in U_n}=\sum_{l\in\mathcal L(m)}\sum_{U_n\cap I_l\not=\varnothing}\sum_{\al_k\in U_n\cap I_l}.
\ee
We use the sets $L_1(m)$ and $L_3(m)$ defined in \eqref{L sets}. The intervals $I_l$, $l\in L_2(m)$ and $I_{\star}$ are omitted, because condition (2) in Definition~\ref{def:gp} and assumption~\ref{ass:S_B}(1) imply $\alst>\al_{\text{mx}}$. 

Given any $l\in\mathcal L(m)$, let $n_l$ be such that $\al_l^*\in U_{n_l}$, i.e. $\al_l^*\asymp \e n_l^2$. By \eqref{rem In sum g22}, we can ignore finitely many $U_n$ on either side of $\al_l^*$. Indeed, there are $O(|m|)$ sets $I_l$, so the total error of these omissions is $O(\e^{1/2})$ provided that \textcolor{black}{$\bt>2$} and the number of such omissions is uniformly bounded with respect to $l$. By the same reason we can ignore those $U_n$ that intersect two successive $I_l$. Hence, we can replace $U_n\cap I_l\not=\varnothing$ with $U_n\subset I_l$ in \eqref{Un 2 sums}. This implies also that the sum with respect to $l$ in \eqref{Un 2 sums} can be replaced by the sum over $l$ such that $\al_l^*\in\Omega_\e$.

Suppose \textcolor{black}{$\bt>(\eta_0+3)/2$.} This condition implies that for every $\al_l^*\in\Omega_\e$ there is a $U_n$ that contains it and each $U_n$ contains no more than one $\al_l^*$. Indeed, from \eqref{ang0 bnd} it follows that $\al_{l_0+1}^*$, the first $\al_l^*\in\Omega_\e$, satisfies $\al_{l_0+1}^*\ge c|m|^{-(\eta+1)}$ for any $\eta>\eta_0$. The assumption on $\bt$ and $|m|\le\e^{-\ga}$ imply $\al_{l_0+1}^*/\e\to \infty$, so $\al_{l_0+1}^*>4\e$ when $\e>0$ is sufficiently small. Also, $\al_{l+1}^*-\al_l^*\asymp 1/|m|$ and $|U_n|\asymp \e n\le c\e^{1/2}$. Hence $1/|m|\ge \e^\ga\ge\e^{1/2}$.

By a preceding argument suffices it to consider only $2\le n\le n_l-2$ and $n_l+2\le n\le N$. 

\noindent
\underline{Intervals $I_l$, $l\in L_1(m)\cup L_3(m)$}. 
Pick any $U_n\subset I_l$ and denote
\be\label{Jn def first C}
J_n:=\frac1{\Delta\al}\int_{U_n} (g_{m,\e}^{(2,2)} h_l)(\al) e\biggl(\frac{\phi_l(\al)}\e\biggr)\dd \al.
\ee
Similarly to \eqref{sum est 1 v2},
\be\label{sum est 2 v1}\begin{split}
&J_n
=\pi[J_{n,1}-J_{n,2}],\
J_{n,1}:=\left. \frac{g_{m,\e}^{(2,2)}(\al)}{\sin(\pi\kappa  \phi_l’(\al))}\right|_{\e n^2}^{\e (n+1)^2},\\
&J_{n,2}:=\int_{U_n}\pa_\al\frac{g_{m,\e}^{(2,2)}(\al)}{\sin(\pi\kappa  \phi_l’(\al))}e\biggl(\frac{\phi_l(\al)}\e\biggr)\dd\al.
\end{split}
\ee
Clearly, $\phmn\asymp|m|$. As was mentioned at the beginning of subsection~\ref{ssec:rem g22 props}, $|g_{m,\e}^{(2,2)}(\al)|\le c\rho(m)\e^{1/2}/n$, $\al\in U_n$. Using the obvious properties 
\be\label{rem aux ineqs}
|\phi_l’(\al)|\ge c\phi_{\text{mn}}^{\prime\prime}|\al-\al_l^*|,\ \al_l^*-\al \asymp\al_l^*-\e n^2,\ \al\in U_n\subset I_l, 
\ee
the integrated term is bounded as follows
\be\label{rem int term 13}
|J_{n,1}|\le c\frac{\rho(m)\e^{1/2}}{n} \frac1{|m||\al_l^*-\e n^2|},\ |n-n_l|\ge 2,\ 2\le n\le N.
\ee
Summing over $n$ and using that $(\al_l^*/\e)^{1/2}\asymp n_l$ gives
\be\label{rem int sum_n 13}\begin{split}
\frac{c\rho(m)\e}{|m|}&\biggl(\sum_{n=2}^{n_l-2}+\sum_{n=n_l+2}^\infty\biggr)\frac1{\e^{1/2}n|\al_l^*-\e n^2|}\\
&\le \frac{c\rho(m)\e^{1/2}}{|m|}\biggl(\int_{\e^{1/2}}^{(\al_l^*)^{1/2}-\e^{1/2}}+\int_{(\al_l^*)^{1/2}+\e^{1/2}}^\infty\biggr) \frac{\dd t}{t|\al_l^*-t^2|}\\
&\le \frac{c\rho(m)\e^{1/2}}{|m|}\frac{\ln(\al_l^*/\e)}{\al_l^*}.
\end{split}
\ee
By \eqref{rem J1 sum} and \eqref{rem aux ineqs},  
\be\label{rem integral 13}\begin{split}
|J_{n,2}|
\le & c\biggl[\max_{\al\in U_n}\biggl(\frac1{|\phi_l’(\al)|}\biggr)\TV(g_{m,\e}^{(2,2)},U_n)\\
&\hspace{2cm}+\max_{\al\in U_n}|g_{m,\e}^{(2,2)}(\al)|\int_{U_n} \biggl|\biggl(\frac1{\phi_l’(\al)}\biggr)’\biggr|\dd\al\biggr]\\
\le & c\frac{\rho(m)\e^{1/2}}{n}\biggl[  \frac1{|m||\al_l^*-\e n^2|} +
\int_{U_n} \frac1{|m|(\al_l^*-\al)^2}\dd\al\biggr]\\
\le & \frac{c\rho(m)}{|m|}\frac{\e^{1/2}}{n}\frac1{|\al_l^*-\e n^2|},
\end{split}
\ee
which coincides with \eqref{rem int term 13}. Hence, summing over $n$, gives \eqref{rem int sum_n 13}.

Add now estimates \eqref{rem int sum_n 13} for all $\al_l^*\in\Omega_\e$:
\be\label{rem sum_l v1}\begin{split}
\frac{c\rho(m)\e^{1/2}}{|m|}\sum_{\al_l^*\in\Omega_\e}\frac{\ln(\al_l^*/\e)}{\al_l^*}.
\end{split}
\ee
Due to the possible rapid decay of $\al_{l_0+1}^*$ as $m\to\infty$, the estimate of the first term in the sum (with $l=l_0+1$) based on \eqref{ang0 bnd} dominates all the other terms and we get that \eqref{rem sum_l v1} is bounded by
\be\label{rem sum_l v2}\begin{split}
c\rho(m)\langle m\kappa\vec\al_0^\perp\cdot x_0\rangle^{-1}\e^{1/2}\ln(1/\e).
\end{split}
\ee
Summing over $1\le|m|\le\e^{-\ga}$ gives $O(\e^{1/2}\ln(1/\e))$ provided that \textcolor{black}{$\bt>\eta_0+1$.} Note that this inequality implies $\bt>(\eta_0+3)/2$, which was assumed earlier, because $\eta_0\ge 1$ (see the paragraph following \eqref{type ineq}).

\subsection{Proof of Lemma~\ref{lem:ineq II B}}
By Lemmas~\ref{lem:aux props rem} and \ref{lem:three cases}, we estimate similarly to \eqref{A-outside}
\be\label{et II st1}\begin{split}
&\sum_k \int_{|\al-\al_k|\le \Delta\al/2}|A_0(\al,\e)-A_0(\al_k,\e)|\dd\al
\le 2\Delta\al \sum_k |A_0(\check\al_k,\e)|
=O(\e^{1/2}),\\
&\check\al_k:=\text{argmax}_{|\al-\al_k|\le \Delta\al/2}|A_0(\al,\e)|,
\end{split}
\ee
where all the sums are over $k$ such that $\al_k\in[-\pi/2,\pi/2]\setminus\Omega_\e$. 
Similarly, by Lemmas~\ref{lem:rem A-to-g2} and \ref{lem:rem g-to-g22}, 
\be\label{et II st2}\begin{split}
&\sum_{\al_k\in\Omega_\e} \int_{|\al-\al_k|\le \Delta\al/2}|A_0(\al,\e)-A_0(\al_k,\e)|\dd\al\\
&\le c\sum_{\al_k\in\Omega_\e} \max_{|\al-\al_k|\le \Delta\al/2}|g_{0,\e}^{(2,2)}(\al,\e)-g_{0,\e}^{(2,2)}(\al_k,\e)|+O(\e^{1/2}\ln(1/\e))\\
&\le c\sum_{U_n\subset\Omega_\e}\TV(g_{0,\e}^{(2,2)},U_n)+O(\e^{1/2}\ln(1/\e)).
\end{split}
\ee
The result follows immediately from \eqref{g_int rem}:
\be\label{rem TV Un}\begin{split}
\sum_{U_n\subset\Omega_\e}\TV(g_{0,\e}^{(2,2)},U_n)
\le c\sum_{1\le n\le \e^{-1/2}}\frac{\e^{1/2}}{n}=O(\e^{1/2}\ln(1/\e)).
\end{split}
\ee

\section{Proofs of Lemmas in section~\ref{sec:caseC}}\label{sec:caseC proofs}

\subsection{Proof of Lemma~\ref{lem:rem2 A-to-g2}}

Transform the expression for $A_m$ (cf. \eqref{recon-ker-v2}) similarly to \eqref{A-simpl} and use  assumption~\ref{ass:S_C}(3):
\be\label{A-simpl-lem}
\begin{split}
A_m(\al,\e)=&\frac1\e\int_{|u|\le a-\de}\int_0^{\e^{-1}H_\e(u)}\tilde\psi_m\left(\e^{-1}R(u,\al)+h(\hat t,u,\al)\right)F(u,\e \hat t)\dd\hat t\dd u,
\end{split}
\ee
where $R$, $F$ and $h$ are defined in \eqref{rem2 gme2}. From the definition of $\Omega$ (see the paragraph preceding assumptions~\ref{ass:S_C}) it follows that $|R(u,\al)|\ge c>0$ when $|u|\le a-\de$ and $\al\in[-\pi/2,\pi/2]\setminus\Omega$. This implies $A_m(\al,\e)=O(\rho(m)\e)$, $\al\not\in\Omega$. Therefore, $\e\sum_m\sum_{\al_k\not\in\Omega}|A_m(\al,\e)|=O(\e)$ if \textcolor{black}{$\bt>1$}. This proves the first line in \eqref{rem2 sums A-g22}. So we assume $\al\in\Omega$.

Clearly,
\be\label{rem2 Rb}
R(u,\al)=b(u,\al)(u-\CA(\al)),\ |u|\le a,\ \al\in\Omega,
\ee
where $\pa_u b$ and $\pa_\al b$ are bounded and $|b|\asymp 1$. These properties of $b$ hold on the set indicated in \eqref{rem2 Rb}.

By \eqref{four-coef-bnd}, we can restrict the domain of $u$ integration in \eqref{A-simpl-lem} (cf. \eqref{rem2 gme2}): 
\be\label{A-simpl v2}
\begin{split}
A_m(\al,\e)=(1/\e)g_{m,\e}^{(1)}(\al)+O(\rho(m)\e^{1/2}),\ \al\in\Omega.
\end{split}
\ee
Summing over $m\in\BZ$ proves the second line in \eqref{rem2 sums A-g22}.
As is easily checked,
\be\label{rem2 g2 bnd app}
\int_0^{\e^{1/2}}(1+(u/\e)^2)^{-1}\dd u=O(\e).
\ee
Together with \eqref{rem2 gme2} and \eqref{rem2 Rb} this proves \eqref{rem2 g2 bnd}. 

Next we estimate $\TV(g_{m,\e}^{(2)},U)$, where $U\subset\Omega$ and $|U|=\e^{1/2}$. Change variable $u\to\zeta=u-\CA(\al)$ and use \eqref{rem2 Rb}:
\be\label{Am small g rem2}\begin{split}
g_{m,\e}^{(1)}(\al)=&\int_{|\zeta|\le\e^{1/2}}\int_0^{H_0(\e^{-1/2}u)}\tilde\psi_m\left(b(u,\al)(\zeta/\e)+h(\hat t,u,\al)\right)\\
&\hspace{3cm}\times F(u,\e \hat t)\dd\hat t\dd \zeta,\ u=\zeta+\CA(\al),\ \al\in U.
\end{split}
\ee
Denote
\be\label{g rem2 v2}\begin{split}
g_{m,\e}^{(2)}(\al_1,\al_2):=&\int_{|\zeta|\le\e^{1/2}}\int_0^{H_0(\tilde u_1)}\tilde\psi_m\left(b(u_2,\al_2)(\zeta/\e)+h(\hat t,u_2,\al_2)\right)\\
&\hspace{3cm}\times F(u_2,\e \hat t)\dd\hat t\dd \zeta,\\ 
\tilde u_1:=&\e^{-1/2}(\zeta+\CA(\al_1)),\ u_2=\zeta+\CA(\al_2),\ \al_1,\al_2\in U.
\end{split}
\ee
Clearly, $g_{m,\e}^{(1)}(\al)\equiv g_{m,\e}^{(2)}(\al,\al)$. Using that $\pa_u b$, $\pa_\al b$, and $\CA’$ are bounded (cf. \eqref{theta deriv min}) gives 
\be\label{delg rem2 v2}\begin{split}
|\pa_{\al_2}g_{m,\e}^{(2)}(\al_1,\al_2)|\le&c\rho(m)\int_0^{\e^{1/2}}\frac{1+(\zeta/\e)}{1+(\zeta/\e)^2}\dd \zeta\\
\le & c\rho(m)\e\ln(1/\e),\ \al_1,\al_2\in U.
\end{split}
\ee
Furthermore,
\be\label{delg rem2 v3}\begin{split}
&|g_{m,\e}^{(2)}(\al_2,\al)-g_{m,\e}^{(2)}(\al_1,\al)|\le c\rho(m)\int_0^{\e^{1/2}}\frac{|H_0(\tilde u_2)-H_0(\tilde u_1)|}{1+(\zeta/\e)^2}\dd \zeta,\\ 
&\tilde u_j:=\e^{-1/2}(\zeta+\CA(\al_j)),\ j=1,2,\ \al,\al_1,\al_2\in U.
\end{split}
\ee
An argument analogous to those in sections~\ref{ssec:step2} and \ref{ssec:rem g22 props} yields
\be\label{TV rem2 pr1}\begin{split}
\TV(g_{m,\e}^{(1)},U) \le c\rho(m)\big(\e^{3/2}\ln(1/\e)+\e\big),
\end{split}
\ee
and \eqref{rem2 TV bnd} is proven.

\begin{remark} Estimate \eqref{TV rem2 pr1} is a key result in the proof of Lemma~\ref{lem:main-res-BC} (case (C)) where we need assumption~\ref{ass:H0}(2).
\end{remark}

\subsection{Proof of Lemma~\ref{lem:rem2 sum g2}}
Similarly to cases (A) and (B), the proof of Lemma~\ref{lem:main-res-BC} in case (C) follows the five steps described after \eqref{extra terms II}. The previous results up to this point complete the first three steps. In this section we complete steps (4) and (5).

\underline{Step (4).} By \eqref{rem2 g2 bnd} and \eqref{rem2 TV bnd}, Lemma~\ref{lem:sum1 est} can be applied to the sums with respect to $\al_k\in\Omega$ in \eqref{rem2 sum g2} for each $m\not=0$, with $g=g_{m,\e}^{(1)}$ given by \eqref{rem2 gme2}. Indeed, replacing the sums by integrals leads to the total error bounded by
\be
c\sum_{m\in\BZ}\rho(m)\bigl[(1+\e|m|)\e^{1/2}+|m|\e\bigr]
=O(\e^{1/2}),\ \textcolor{black}{\bt>2.}
\ee
Furthermore, \eqref{rem2 g2 bnd} implies that, for any $m$, in sum \eqref{rem2 sum g2} we can ignore all $\al_k$ in $O(|m|)$ intervals of length $O(\e^{1/2})$ each as long as \textcolor{black}{$\bt>2$.}  

\underline{Step (5).} In this step we assume $m\not=0$. Pick any $I_l$, $l\in \mathcal L(m)$, and remove an $O(\e^{1/2})$ neighborhood of $\al_l^*$ from it. More precisely, we replace $I_l$ by two intervals $I_l^-:=[\al_{l-(1/2)}^*,\al_l^*-\e^{1/2}]$ and $I_l^+:=[\al_l^*+\e^{1/2},\al_{l+(1/2)}^*]$. The two intervals are analogous, so we always use $I_l^+$ in the rest of this step. 

Consider the integrals
\be\label{Jn def first rem2}
J_l:=\frac1{\Delta\al}\int_{I_l^+} (g_{m,\e}^{(1)} h_l)(\al) e\biggl(\frac{\phi_l(\al)}\e\biggr)\dd \al.
\ee
Similarly to \eqref{sum est 1 v2} and \eqref{sum est 2 v1},
\be\label{sum est 2 v1 rem2}\begin{split}
&J_l
=\pi[J_{l,1}-J_{l,2}],\
J_{l,1}:=\left. \frac{g_{m,\e}^{(1)}(\al)}{\sin(\pi\kappa  \phi_l’(\al))}\right|_{\al_l^*+\e^{1/2}}^{\al_{l+(1/2)}^*},\\
&J_{l,2}:=\int_{I_l^+}\pa_\al\bigg(\frac{g_{m,\e}^{(1)}(\al)}{\sin(\pi\kappa  \phi_l’(\al))}\bigg) e\biggl(\frac{\phi_l(\al)}\e\biggr)\dd\al.
\end{split}
\ee
By \eqref{rem2 g2 bnd}, 
\be\label{rem2 int term 13}
|J_{l,1}|\le c \rho(m)\e/(\e^{1/2}\phmn(I_l^+)).
\ee
Break up $I_l^+$ into $O(\e^{-1/2})$ intervals of length $O(\e^{1/2})$ each. Then, by \eqref{rem2 g2 bnd} and \eqref{rem2 TV bnd}, 
\be\label{rem2 integral 13}\begin{split}
|J_{l,2}|&\le  c\frac{\rho(m)\e}{\phmn(I_l^+)}\biggl[ \sum_{1\le |n|\le O(\e^{-1/2})} \frac1{\e^{1/2}n}+
\int_{|\al_l^*-\al|\ge \e^{1/2}} \frac{\dd\al}{(\al_l^*-\al)^2}\biggr]\\
&\le c \frac{\rho(m)\e^{1/2}\ln(1/\e)}{\phmn(I_l^+)}.
\end{split}
\ee
Clearly, $\phmn(I_l^+)\ge \phmn(I_l)$. A bound on the latter quantity is in \eqref{phmn ineq}. In case (C), we can consider all $l\in\mathcal L(m)$ together because only the point $\alst$ is special ($\al=0$ is no longer special). Therefore, summing over $l\in\mathcal L(m)$ is essentially equivalent to summing over $l\in L_2(m)$. 
Using \eqref{sum 21} with $n=1$ gives the same estimate as in \eqref{exc2 Jn bnd}, which requires \textcolor{black}{$\bt>(\eta_\star+1)/2$.}

Suppose now $\al_k\in I_{\star}$. As is easily seen, cases (B) and (C) are equivalent when $\al_k\in I_{\star}$. 
In particular, the bounds in \eqref{exc2 g_e ass} and \eqref{rem2 g2 bnd}, \eqref{rem2 TV bnd} are the same. 
Hence the bound \eqref{exc II}, which requires \textcolor{black}{$\bt>(\eta_\star/2)+1$,} still applies. 

\subsection{Proof of Lemma~\ref{lem:rem2 sum g3}}\label{ssec:rem2 g3}
Using \eqref{A-simpl v2} with $m=0$ and \eqref{rem2 TV bnd} and arguing similarly to the proof of Lemmas~\ref{lem:ineq II}, \ref{lem:ineq II B} proves \eqref{extra terms II}. 

\section{Estimation of model integrals and a sum}\label{sec:model ints}

Throughout this section each integral being estimated is denoted $J$.

\subsection{Proof of Lemma~\ref{lem:first int}}\label{ssec:first int}

Change variables $\mu=\e^{-1/2}\al$, $t=\e^{-1/2}\theta$. Then
\be\begin{split}
J\le & c\e^{1/2}\biggl(\int_{|t|\le 1}+\int_{|t-\mu|\le 1}+\int_{|t|\ge 1,|t-\mu|\ge 1}\biggr)\frac{\dd t}{1+(t(t-\mu))^2}.
\end{split}
\ee
It is clear that the first two integrals are $O((1+|\mu|)^{-1})$ each. To estimate the third one, suppose, for example, $\al>0$ (and $\mu>1$):
\be\begin{split}
\int_{-\infty}^{-1}\frac{\dd t}{1+(t(t-\mu))^2}\le \biggl(\int_1^\mu+\int_\mu^\infty\biggr) \frac{\dd t}{t^2(t+\mu)^2}\le c(1+\mu)^{-2}.
\end{split}
\ee
The integral over $t\in [\mu+1,\infty)$ admits the same bound. The last integral is estimated as follows:
\be\begin{split}
\int_{1}^{\mu-1}\frac{\dd t}{1+(t(t-\mu))^2}\le c\int_1^{\mu/2} \frac{\dd t}{t^2\mu^2}\le c\mu^{-2}.
\end{split}
\ee

\subsection{Proof of \eqref{J22ab}}\label{ssec:J22ab}

From the definition of $J_{2,2}^a$ in \eqref{integral br} we have
\be\label{J22a 1}\begin{split}
J_{2,2}^a:=&\sum_{n=1}^{N-1}\frac{1}{n}\frac{1}{1+|n+\tilde p|}.
\end{split}
\ee
If $\tilde p>0$, we have
\be\label{J22a 1p}\begin{split}
J_{2,2}^a\le &c\int_{1}^{\infty}\frac{\dd x}{x(x+\tilde p)}=O(\ln(\tilde p)/\tilde p).
\end{split}
\ee
If $\tilde p<0$,
\be\label{J22a 2}\begin{split}
J_{2,2}^a\le &c\biggl[\int_1^{|\tilde p|-1}\frac{\dd x}{x(|\tilde p|-x)}+\int_{|\tilde p|-1}^{|\tilde p|+1}\frac{\dd x}{|\tilde p|}+\int_{|\tilde p|+1}^{\infty} \frac{\dd x}{x(x-|\tilde p|)}\biggr]\\
=&O(\ln|\tilde p|/|\tilde p|).
\end{split}
\ee

Consider now $J_{2,2}^b$. Suppose $\tilde p>0$. From \eqref{aux ints} and \eqref{integral br} we obtain by changing variable $\al\to x=(\al-\al_l^*)/\e^{1/2}$:
\be\label{J22b 1}\begin{split}
J_{2,2}^b=\int_1^\infty\frac{\dd x}{(1+|\tilde p+x|)x^2}\le &\int_1^\infty\frac{\dd x}{(\tilde p+x)x^2}=O(1/\tilde p).
\end{split}
\ee
If $\tilde p<0$, then
\be\label{J22b 2}\begin{split}
J_{2,2}^b\le &c\biggl[\int_1^{|\tilde p|-1}\frac{\dd x}{x^2(|\tilde p|-x)}+\int_{|\tilde p|-1}^{|\tilde p|+1}\frac{\dd x}{\tilde p^2}+\int_{|\tilde p|+1}^{\infty} \frac{\dd x}{x^2(x-|\tilde p|)}\biggr]\\
=&O(1/|\tilde p|).
\end{split}
\ee

\subsection{Proofs of \eqref{model int add 1} and \eqref{model der int}}\label{ssec: two sep ints}

To prove \eqref{model int add 1}, we change variables $t=\e^{-1/2}s$ and denote $\tal=\e^{-1/2}\al$. This gives
\be\begin{split}
J=\int_{|t|\le 1}\frac{|t|}{1+(t(\tal+t))^2}\dd t\le c\int_0^1 \frac{t}{1+\tal^2t^2}\dd t,
\end{split}
\ee
and the desired assertion follows. 

Similarly, use $t=\e^{-1/2}\theta$ and denote $\tal=\e^{-1/2}\al$ in \eqref{model der int} to get
\be\label{model der int 1}\begin{split}
J=\biggl(\int_{-\infty}^{\tal-1}+\int_{\tal+1}^\infty\biggr)\frac{|t|}{1+(t(t-\tal))^2}\dd t.
\end{split}
\ee
Suppose, without loss of generality, that $\tal\to+\infty$. The second integral is $O(\tal^{-1})$. To estimate the first integral, we split it into three, with the domains of integration $(-\infty,-1]$, $[-1,1]$, and $[1,\tal-1]$. The first two are easily seen to be $O(\ln(\tal)/\tal^2)$. Up to a constant factor, the third integral does not exceed the sum
\be\label{model der int1}\begin{split}
&\int_{1}^{(\tal/2)-1}\frac{t}{1+t^2\tal^2}\dd t+\int_{(\tal/2)-1}^{\tal-1}\frac{\tal}{1+\tal^2(t-\tal)^2}\dd t\\
&\le O(\tal^{-2}\ln\tal)+O(\tal^{-1})=O(\tal^{-1}).
\end{split}
\ee

\subsection{Proof of Lemma~\ref{lem:three cases}}\label{ssec:three cases}

Consider the first case. Set $\mu=\al^{1/2}$. Restricting the integration to $[0,a]$ gives
\be\begin{split}
J\le & c \int_0^a\frac{\dd\theta}{(1+\e^{-1}(\mu-\theta)\mu)^2}\le c\frac{\e}\mu\biggl[\frac1{1+\e^{-1}(\mu-a)\mu}-\frac1{1+\e^{-1}\mu^2}\biggr]\\
\le & c\frac1{(1+\e^{-1}(\mu-a)\mu)(1+\e^{-1}\mu^2)}\le c\frac\e{1+\e^{-1}(\mu-a)\mu},
\end{split}
\ee
where we have used that $1/\mu^2=1/\al\asymp 1$. Using that $(\mu-a)\mu\asymp\mu^2-a^2$ finishes the proof.

In the second case, 
\be\begin{split}
J\le & c \biggl(\int_0^\mu+\int_\mu^\infty\biggr)\frac{\dd\theta}{1+\e^{-2}(\theta^2-\mu^2)^2}
\le c\biggl[\e^{1/2}+\int_0^\infty\frac{\dd\theta}{1+(\theta^2/\e)^2}\biggr]\\
=&O(\e^{1/2}).
\end{split}
\ee
The last inequality follows by noticing that 
\be
1+(\theta^2-\mu^2)^2/\e^2\asymp 1+(\theta-\mu)^4/\e^2 
\ee
if $\mu=O(\e^{1/2})$, $\theta\ge\mu$ and changing variables $\theta\to\theta-\mu$.

To prove the third case, set $\mu=(1+|\al|/\e)^{1/2}$ and change variables $t=\e^{-1/2}\theta$:
\be
J\le c\e^{1/2}\int_0^\infty \frac{\dd t}{(\mu^2+t^2)^2}\le c\e^{1/2}/\mu^3.
\ee

\subsection{Proof of \eqref{rem mod2}}\label{ssec:Phi1_2}
Begin with $\Phi_1$. Setting $t=\e^{-1/2}\theta$ and $\mu=(\al/\e)^{1/2}$ in \eqref{rem mod1} gives $\mu\ge 2$ and
\be\label{rem mod1_1}\begin{split}
\Phi_1(\al)&\le c\e^{1/2}\int_{|t^2-\mu^2|\ge \mu}\frac{\dd t}{(t^2-\mu^2)^2}
\le c\frac{\e^{1/2}}{\mu^3}\biggl(1+\int_{1/\mu\le |s^2-1|\le 1/2}\frac{\dd s}{(s^2-1)^2}\biggr)\\
&\le c\frac{\e^{1/2}}{\mu^3}\biggl(1+\int_{1/\mu\le |v-1|\le 1/2}\frac{\dd v}{(v-1)^2}\biggr)
\le c \frac{\e^{1/2}}{\mu^2}.
\end{split}
\ee
Using the same substitutions in $\Phi_2$ we find
\be\label{rem mod2_1}\begin{split}
\Phi_2(\al)&\le c\e^{1/2}\int_{|t^2-\mu^2|\le \mu}\frac{\dd t}{1+(t^2-\mu^2)^2}
\le c\e^{1/2}\int_{|t-\mu|\le c}\frac{\dd t}{1+\mu^2(t-\mu)^2}\\
&\le c\e^{1/2}/\mu.
\end{split}
\ee

\bibliographystyle{abbrv}
\bibliography{My_Collection}
\end{document}